\documentclass{amsart}
\usepackage[dvips]{graphicx}
\usepackage{color}
\usepackage{float}
\usepackage{epsfig}
\usepackage{enumerate}
\usepackage{url}
\usepackage{multirow}
\usepackage{bbm}
\usepackage{amsmath,amssymb,amsfonts}
\usepackage{latexsym}

\newtheorem{thm}{\bf{Theorem}}[section]

\newtheorem{rem}[thm]{\bf{Remark}}

\begin{document}

\title[]{Multistage stochastic programs with a random number of stages: dynamic programming equations, solution methods, and application to portfolio selection}

\maketitle 

\begin{center}
Vincent Guigues\\
School of Applied Mathematics, FGV\\
Praia de Botafogo, Rio de Janeiro, Brazil\\ 
{\tt vincent.guigues@fgv.br}
\end{center}

\date{}

\begin{abstract} We introduce the class of multistage stochastic optimization problems with a random number of stages.
For such problems, we show how to write dynamic programming equations and  how to solve these equations using
the Stochastic Dual Dynamic Programming algorithm. Finally, we consider a portfolio selection problem
over an optimization period of random duration. For several instances of this problem, we show the 
gain obtained using a policy that
takes the randomness of the number of stages into account
over a policy built taking a fixed number of stages (namely the maximal possible number of stages).\\
\end{abstract}



\section{Introduction}

Multistage Stochastic Programs (MSPs) are common in finance and in many areas of engineering, see for instance \cite{shadenrbook} and references therein.
These models are useful when a sequence of decisions has to be taken over an optimization period of $T$ stages knowing that
these decisions must satisfy almost surely random constraints and induce some random costs \cite{shadenrbook, aloispflugbook, romischpflugbook, shapiroruszbook, philpshapiro}.
To the best of our knowledge all MSPs considered so far in the literature have a known (finite in general) number of stages.
However, for many applications the number of stages, i.e., 
the {\em{real}} optimization period, is not known in advance. It is easy to name a few of these applications:
\begin{itemize}
\item A company may want to determine
optimal investments over its lifetime \cite{mulveyshetty, yliu, singhfaircloth, mulveykim}. In this situation, the optimization period ends when the company disappears either
because it goes bankrupt, or because it is bought by another company, or because it decides to stop its activities \cite{lifetimecompanies2015}.
These three stopping times, which determine the number of stages $T$, are indeed random.
\item A fund manager can decide to stop his investments when the fund reaches a given target. This stopping time is again
random, depending on the random returns of the investments and of the  investment strategy.
\item Multistage stochastic portfolio optimization \cite{dantinfanger, bentalnemmarg, FILE}: an individual may invest his money in financial assets until his death or until he obtains a given amount used
for covering some expense \cite{bradleycrane}. Again, both stoppping times are random.
\item An hedge fund may have to deal with longevity risk \cite{andrewmichael, financebook, almbook, kusyziemba}, with payout ratios for a given set of individuals spreading over
random time windows .
\end{itemize}
The examples above are examples in finance but in many other areas, for instance logistics \cite{chendemello}, power management \cite{pereirapinto, pereira, varedf, robprodman},
and capacity planning and expansion \cite{chenlitirupati, davisdempstersethi, jinryanwatsonwood}, MSPs with a random optimization period could be useful, especially for 
long-term optimization \cite{flachbarrosopereira} when the optimization period depends on the lifetime of individuals or companies.

It is therefore natural to consider multistage stochastic programs having a random number of stages.
The study of these problems passes through two successive steps: (i) a modelling step to define a policy
and (ii) an algorithmic step to build a policy, i.e., 
a solution method allowing us to compute decisions on any realization
of the uncertainty. Guided by the fact that there exist efficient solution methods for MSPs based on dynamic programming equations
(for instance Stochastic Dual Dynamic Programming (SDDP) \cite{pereira} and Approximate Dynamic Programming (ADP) \cite{powellbook}),
our goal for (i) is to write dynamic programming equations for a multistage stochastic program with a random number of stages.
The paper is organized as follows.
In Section \ref{dpequations}, we define multistage stochastic optimization problems with a random number of
stages and show how to write dynamic programming equations for these problems. We show in particular, that compared to the case
where the number of stages is fixed, two new features appear: first
we need to add an extra state variable, denoted by $D_{t-1}$ for stage $t$, allowing us to know if the optimization period is already over or not
and second for each stage instead of just one cost-to-go function we have two cost-to-go functions, one if the optimization period
is already over (this is the null function) and another one when there remains additional stages for the optimization period.
In Section \ref{dpmoregeneral}, we write dynamic programming equations for a slightly larger class of MSPs, still having a random number of stages, but
where the cost function for the last (random) stage and the cost functions for the remaining stages are taken from two sets
of functions. We provide a portfolio selection model as an example of such problems.

In the case when the underlying stochastic process $\xi_t$ neither depends on its past $(\xi_1,\ldots,\xi_{t-1})$ nor on $D_t$
and $D_t$ only depends on $D_{t-1}$, we detail in Section \ref{sec:sddp} the SDDP algorithm to solve the dynamic programming equations written in Section \ref{dpequations}.
This variant of SDDP, called SDDP-TSto, is very similar to the variants of SDDP presented in \cite{philpmatos}
where the underlying stochastic process depends on a Markov Chain (process $(D_t)$ in our case) and a value function is used 
for each stage and in each state of the Markov chain.
Finally, in Section \ref{sec:numexp}, we consider a portfolio problem with a random optimization period and
the corresponding dynamic programming equations, given in Section \ref{dpmoregeneral}. 
We detail the SDDP algorithm applied to these equations and
present the results of numerical tests which compare for several instances the performance of a policy that takes the randomness of the number of stages 
into account with the performance of a policy built taking a fixed value for the number of stages, namely the maximal possible value.

\section{Writing dynamic programming equations for multistage stochastic programs with a random number of stages}\label{dpequations}

Consider a  risk-neutral  multistage stochastic optimization problem with $T_{\max}$ known stages of form
\begin{equation} \label{pbsto}
\begin{array}{l}
\displaystyle{\inf} \; \mathbb{E}_{\xi_2,\ldots,\xi_{T_{\max}}}\Big[ \displaystyle{\sum_{t=1}^{T_{\max}}} \; f_{t} (x_{t}, x_{t-1}, \xi_t)\Big]\\
x_t \in X_t( x_{t-1}, \xi_t )\;\mbox{a.s.},  x_{t} \;\mathcal{F}_t\mbox{-measurable, }t=1,\ldots,T_{\max},
\end{array}
\end{equation}
where  $x_0$ is given, $\xi_1$ is deterministic,  $(\xi_t)_{t=2}^{T_{\max}}$ is a stochastic process, $\mathcal{F}_t$ is the sigma-algebra
$\mathcal{F}_t:=\sigma(\xi_j, j\leq t)$, and $X_t( x_{t-1}, \xi_t )$ is a subset of $\mathbb{R}^n$. In the objective, the expectation is computed with respect to the distribution
of $\xi_2,\ldots,\xi_{T_{\max}}$.
We assume that the problem above is well defined. We will come back in Section \ref{sec:sddp}
to the necessary assumptions for problem \eqref{pbsto} to be well defined and to apply SDDP as a solution method.

Our goal in this section is to define multistage stochastic optimization problems where the number of stages is not fixed
($T_{\max}$ in \eqref{pbsto}) anymore but is stochastic, and to derive dynamic programming equations, under several
assumptions, for such problems.

We will assume that
\begin{itemize}
\item[(H1)] the number of stages $T$ is a discrete random variable taking values in $\{2,\ldots,T_{\max}\}$.
\end{itemize}
The number of stages $T$, or stopping time, induces the Bernoulli process
$D_t,\;t=1,\ldots,T_{\max}$ (a ``Death" process), where $D_t=\mathbbm{1}_{T>t}$ is the indicator
of the event $\{T >t\}$:
\begin{equation}\label{indicatordef}
D_t=\mathbbm{1}_{T>t} = \left\{
\begin{array}{ll}
0 & \mbox{if the optimization period ended at }t \mbox{ or before},\\
1 & \mbox{otherwise.}\\
\end{array}
\right.
\end{equation}
Therefore $T$ can be written as the following function of process $(D_t)$:
\begin{equation}\label{TfuncDt}
T= \min \Big\{ 1 \leq t  \leq T_{\max} : D_t = 0  \Big\}. 
\end{equation}
Clearly, $D_t, t=1,\ldots,T_{\max}$, are dependent random variables and
the distribution of $D_t$ given $D_{t-1}$  is known as long as the distribution of $T$ is known.
More precisely, since we have at least 2 stages, $D_1$ takes value 1 with probability 1.
Next, denoting $p_t = \mathbb{P}(T=t)$ and
$q_t = \mathbb{P}( D_t = 0 | D_{t-1} = 1 )$, we have
$q_2 = \mathbb{P}(T=2)=p_2$,
and for $t \in \{2,\ldots,T_{\max}\}$ we get
$$
p_t  =  \mathbb{P}(T=t)=\mathbb{P}(D_2 = 1 ; D_3 = 1; \ldots; D_{t-1}= 1 ; D_t =0 )
 =  q_t \prod_{k=2}^{t-1} (1-q_k).
$$
Therefore transition probabilities $q_t,t=2,3,\ldots,T_{\max}$, can be computed using the recurrence
\begin{equation}\label{qtformula}
\displaystyle q_t =  \displaystyle \frac{p_t}{\displaystyle \prod_{k=2}^{t-1} (1 - q_k )},t=3,\ldots,T_{\max},
\end{equation}
starting from $q_2=p_2$ (note that $q_{T_{\max}}=1$). 

By definition of $D_t$, we also have $\mathbb{P}(D_t=0 | D_{t-1} = 0 ) =1$
or equivalently $\mathbb{P}(D_t=1 | D_{t-1} = 0 ) =0$.

We represent on the top left plot of Figure \ref{figs1} 
the scenario tree of the realizations of $D_1,D_2,\ldots,D_{T_{\max}}$ (for an example where $T_{\max}=5$), as well as the transition probabilities
between the nodes of this scenario tree. In this scenario tree, realizations are indicated at the nodes
of the tree while transition probabilities between two nodes are given above the arrow linking these nodes.
Observe that for a node with label 0 all future transition probabilities are 1 and all descendant nodes
have label 0.

\begin{figure}
\centering
\begin{tabular}{l}
\input{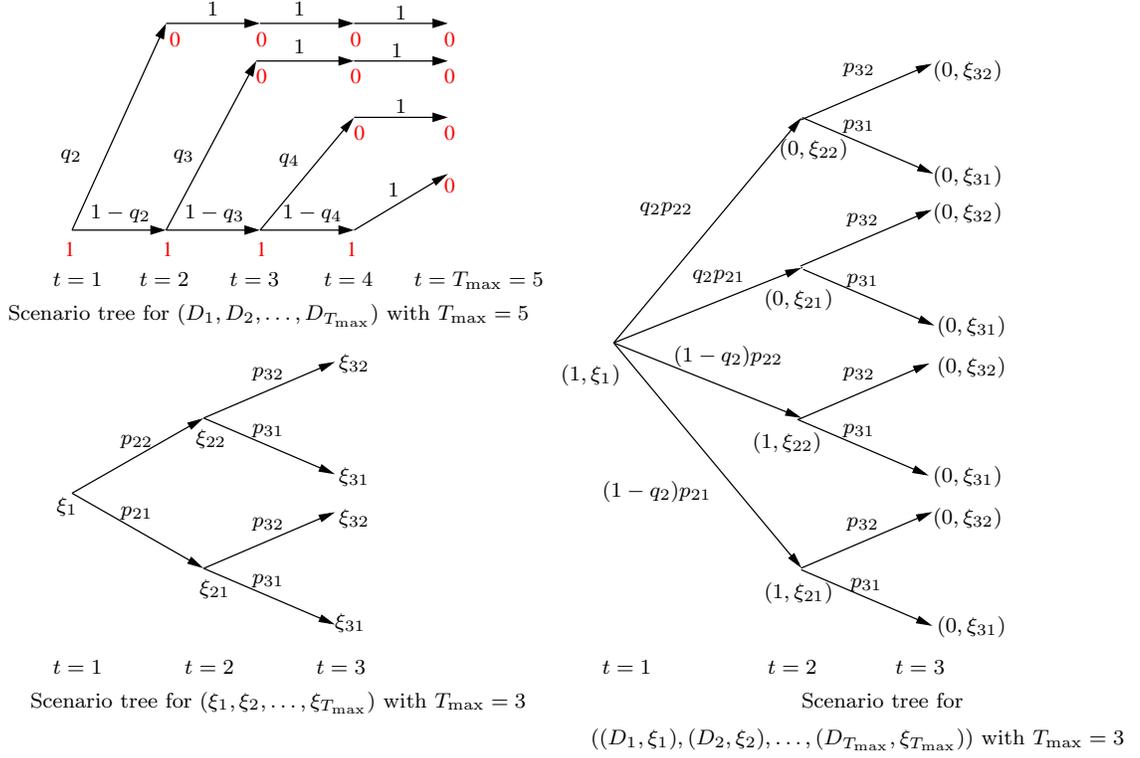}
\end{tabular}
\caption{Scenario trees (assuming that $\xi_t$ does not depend on $(\xi_{[t-1]}, D_t)$).}
\label{figs1}
\end{figure}

In the case when the number of stages is stochastic, the decision $x_t$ for stage $t$
is not only a function of the history $\xi_{[t]}=(\xi_1,\xi_2,\ldots,\xi_t)$ of process $(\xi_t)$, as in \eqref{pbsto},
but also depends on the history of process $(D_t)$. Therefore, we come to the following definition of a risk-neutral
multistage stochastic optimization problem with a random number $T$ of stages:
\begin{equation} \label{defpbstobis}
\begin{array}{l}
\displaystyle{\inf} \; \mathbb{E}_{\xi_2,\ldots,\xi_{T_{\max}},D_2,\ldots,D_{T_{\max}}}\Big[ \displaystyle{\sum_{t=1}^{T}} \; f_{t} (x_{t}, x_{t-1}, \xi_t)\Big]\\
x_t \in X_t( x_{t-1}, \xi_t ) \mbox{ a.s.},\;  x_{t} \;{\overline{\mathcal{F}}}_t\mbox{-measurable, }t=1,\ldots,T_{\max},
\end{array}
\end{equation}
where ${\overline{\mathcal{F}}}_t$ is the sigma-algebra 
\begin{equation}\label{deffttr}
{\overline{\mathcal{F}}}_t =\sigma(\xi_j,D_j,j\leq t)
\end{equation}
and where $T$ is the function of $(D_t)$ given by \eqref{TfuncDt}. Note that in the objective of \eqref{defpbstobis} the expectation
is computed with respect to the distribution of $\xi_2,\ldots,\xi_{T_{\max}},D_2,\ldots,D_{T_{\max}}$.
Plugging  \eqref{TfuncDt} into \eqref{defpbstobis}, problem \eqref{defpbstobis} can be written 
\begin{equation} \label{defpbstoter}
\begin{array}{l}
\displaystyle{\inf} \; \mathbb{E}_{\xi_2,\ldots,\xi_{T_{\max}},D_2,\ldots,D_{T_{\max}}}\Big[ \displaystyle{\sum_{1 \leq t \leq \min \{ 1 \leq \tau  \leq T_{\max} : D_{\tau} = 0  \}}} \; f_{t} (x_{t}, x_{t-1}, \xi_t)\Big]\\
x_t \in X_t( x_{t-1}, \xi_t )\;\mbox{a.s.},\;  x_{t} \;{\overline{\mathcal{F}}}_t\mbox{-measurable, }t=1,\ldots,T_{\max}.
\end{array}
\end{equation}
To write dynamic programming equations for \eqref{defpbstoter} 
we now define the state vectors.
The state vector at stage $t$ is given by $x_{t-1}$ (decision taken at the previous stage)
and the relevant history of processes $(\xi_t)$ and $(D_t)$.
Though all the history $D_{[t-1]}=(D_1,\ldots,D_{t-1})$ of process $(D_t)$ until stage $t-1$ may be necessary,
we argue that it is enough to put in the state vector for stage $t$ past value
$D_{t-1}$ of $(D_t)$. 
Indeed, 
\begin{itemize}
\item if $D_{t-1}=1$ then the whole history of $(D_t)$ until $t-1$ is known: we know
that $D_j=1$ for $1 \leq j \leq t-1$;
\item on the other hand, if $D_{t-1}=0$ then  
whatever the history of $(D_t)$ until $t-1$, we know that the cost function is
null for stage $t$ because the optimization period ended at $t-1$ or before.
\end{itemize}

Consequently the state vector at stage $t$ is $(x_{t-1}, \xi_{[t-1]}, D_{t-1})$ and we introduce for each stage  $t=2,\ldots,T_{\max}$, two functions:
\begin{itemize}
\item $\mathfrak{Q}_t$ such that  $\mathfrak{Q}_t(  x_{t-1}, \xi_{[t-1]}, D_{t-1} , \xi_t    , D_t     )$
is the optimal mean cost from $t$ on starting at $t$ from state $(x_{t-1}, \xi_{[t-1]}, D_{t-1})$ and knowing the values 
$\xi_t$ and $D_t$ of processes $(\xi_t)$ and $(D_t)$ at $t$;
\item $\mathcal{Q}_{t}$ given by  
\begin{equation}\label{dp4}
\mathcal{Q}_{t}(x_{t-1}, \xi_{[t-1]}, D_{t-1} )=\mathbb{E}_{\xi_t, D_t } \Big[ \mathfrak{Q}_t (  x_{t-1}, \xi_{[t-1]}, D_{t-1} , \xi_t    , D_t ) |  D_{t-1} , \xi_{[t-1]}   \Big],
\end{equation}
i.e.,  $\mathcal{Q}_{t}(x_{t-1}, \xi_{[t-1]}, D_{t-1})$ is  
the optimal mean cost from $t$ on
 starting at $t$ from state $(x_{t-1}, \xi_{[t-1]}, D_{t-1})$.
\end{itemize}
We also set $\mathcal{Q}_{T_{\max}+1}(x_{T_{\max}}, \xi_{[T_{\max}]}, D_{T_{\max}}) \equiv 0$. With these definitions, clearly 
for $t=2,\ldots, T_{\max}$,
we have 
\begin{equation}\label{dp0}
\mathcal{Q}_{t}(x_{t-1}, \xi_{[t-1]}, 0 )=0.
\end{equation}
Next, for $t=2,\ldots, T_{\max}$, functions $\mathcal{Q}_{t}(\cdot, \cdot, 1 )$
satisfy the recurrence 
\begin{equation}\label{dp1}
\mathcal{Q}_{t}(x_{t-1}, \xi_{[t-1]}, 1 )=\mathbb{E}_{\xi_t, D_t } \Big[ \mathfrak{Q}_t (  x_{t-1}, \xi_{[t-1]}, 1 , \xi_t    , D_t ) |  D_{t-1}=1 , \xi_{[t-1]}   \Big]
\end{equation}
where 
\begin{equation}\label{dp2}
\mathfrak{Q}_t(  x_{t-1}, \xi_{[t-1]}, 1 , \xi_t    , 0     ) = 
\left\{ 
\begin{array}{l}
\inf_{x_t} \;f_t(x_t, x_{t-1}, \xi_t ) \\ 
x_t \in X_t(x_{t-1}, \xi_t ),
\end{array}
\right.
\end{equation}
and 
\begin{equation}\label{dp3}
\mathfrak{Q}_t(  x_{t-1}, \xi_{[t-1]}, 1 , \xi_t    , 1     ) = 
\left\{ 
\begin{array}{l}
\inf_{x_t} \;f_t(x_t, x_{t-1}, \xi_t ) + \mathcal{Q}_{t+1}(x_{t}, \xi_{[t-1]}, \xi_t , 1 ) \\ 
x_t \in X_t(x_{t-1}, \xi_t ). 
\end{array}
\right.
\end{equation}
The reasons for equations \eqref{dp1}-\eqref{dp3} are clear: 
\begin{itemize}
\item $\mathfrak{Q}_t(  x_{t-1}, \xi_{[t-1]}, 1 , \xi_t    , 0     )$
is the optimal mean cost from stage $t$ on, knowing that the optimization period ends at $t$
and that $\xi_t$ is the value of process $(\xi_t)$ at stage $t$. Therefore it is obtained minimizing the immediate 
stage $t$ cost while satisfying the constraints for stage $t$.
\item $\mathfrak{Q}_t(  x_{t-1}, \xi_{[t-1]}, 1 , \xi_t    , 1     )$
is the optimal mean cost from stage $t$ on, knowing that the optimization period continues after $t$
and that $\xi_t$ is the value of process $(\xi_t)$ at stage $t$. Therefore it is obtained minimizing the immediate 
stage $t$ cost plus the future optimal mean cost (which is $\mathcal{Q}_{t+1}(x_{t}, \xi_{[t]}, D_t )
=\mathcal{Q}_{t+1}(x_{t}, \xi_{[t-1]}, \xi_t , 1 )$  since $D_{t}=1$)
while satisfying the constraints for stage $t$.
\end{itemize}
We observe that  equations \eqref{dp0}-\eqref{dp3} can be written under the  following compact form: for $t=2,\ldots, T_{\max}$,
\begin{equation}\label{dp4}
\mathcal{Q}_{t}(x_{t-1}, \xi_{[t-1]}, D_{t-1} )=\mathbb{E}_{\xi_t, D_t } \Big[ \mathfrak{Q}_t (  x_{t-1}, \xi_{[t-1]}, D_{t-1} , \xi_t    , D_t ) |  D_{t-1} , \xi_{[t-1]}   \Big]
\end{equation}
where 
\begin{equation}\label{dp5}
\mathfrak{Q}_t(  x_{t-1}, \xi_{[t-1]}, D_{t-1} , \xi_t    , D_{t}     ) = 
\left\{ 
\begin{array}{l}
\inf_{x_t} \;D_{t-1} f_t(x_t, x_{t-1}, \xi_t ) + \mathcal{Q}_{t+1}( x_{t}, \xi_{[t]}, D_t  )   \\ 
x_t \in X_t(x_{t-1}, \xi_t ).
\end{array}
\right.
\end{equation}
Setting $D_0=1$, recalling that $D_1=1$ and that ${\overline{\mathcal{F}}}_t$ is given by \eqref{deffttr}, 
it is straighforwardly seen that the optimal value of \eqref{defpbstoter} can be expressed as
\begin{equation}\label{optvalTstomu}
\begin{array}{l}
\displaystyle{\inf}_{x_1} \; D_{0} f_{1} (x_{1}, x_{0}, \xi_1) +\mathcal{Q}_2(x_1, \xi_1 , D_1 )  \\
x_1 \in X_1( x_{0}, \xi_1 ),
\end{array}
\end{equation}
and that \eqref{dp4}-\eqref{dp5} are dynamic programming equations for the problem 
\begin{equation} \label{reformpb1}
\begin{array}{l}
\displaystyle{\inf} \; \mathbb{E}_{\xi_2,\ldots,\xi_{T_{\max}, D_2,\ldots,D_{T_{\max}}}}\Big[ \displaystyle{\sum_{t=1}^{T_{\max}}} \; D_{t-1} f_{t} (x_{t}, x_{t-1}, \xi_t)\Big]\\
x_t \in X_t( x_{t-1}, \xi_t )\;\mbox{a.s.},\;x_{t} \;{\overline{\mathcal{F}}}_t \mbox{-measurable, }t=1,\ldots,T_{\max},
\end{array}
\end{equation}
which is an equivalent reformulation of  \eqref{defpbstoter}.

The impact of the randomness of $T$ on the dynamic programming equations is clear from reformulation \eqref{reformpb1}
of problem \eqref{defpbstoter}. In this reformulation, the number of stages is fixed and known: it is the maximal possible
number of stages $T_{\max}$ for $T$. Therefore, it takes the form of a ``usual" multistage stochastic optimization problem
where random variable $T$ was replaced by the interstage dependent random process $(D_t)$ and the cost
function $f_t$ at stage $t$ was replaced   by the random cost
function $D_{t-1} f_t$. Indeed, when the optimization period ended at $t-1$ or before, the 
cost function is null for stage
$t$ or equivalently can be expressed as $D_{t-1} f_t$ since $D_{t-1}=0$ in this case.
On the other hand, if the optimization period did not end at $t-1$ then 
$D_{t-1}=1$  and again the cost function for stage $t$ can be expressed as $D_{t-1} f_t$ ($= f_t$ in this case). 

Note that in these equations, $(\xi_t, D_t)$ can depend on $\xi_{[t-1]}$.
Clearly $D_t$ depends on $D_{t-1}$ but $(\xi_t, D_t)$ can be independent  on $\xi_{[t-1]}$.
In this situation, $\xi_{[t-1]}$ is not needed in the state vector at $t$ and the dynamic programming equations
simplify as follows: $\mathcal{Q}_{T_{\max}+1}(x_{T_{\max}}  ,   D_{T_{\max}}     ) \equiv 0$ and for $t=2,\ldots,T_{\max}$, we have 
\begin{equation}\label{dp6}
\mathcal{Q}_{t}(x_{t-1}, D_{t-1} )=\mathbb{E}_{\xi_t, D_t } \Big[ \mathfrak{Q}_t (  x_{t-1}, D_{t-1} , \xi_t    , D_t ) |  D_{t-1}   \Big]
\end{equation}
where 
\begin{equation}\label{dp7}
\mathfrak{Q}_t(  x_{t-1}, D_{t-1} , \xi_t    , D_{t}     ) = 
\left\{ 
\begin{array}{l}
\inf_{x_t} \;D_{t-1} f_t(x_t, x_{t-1}, \xi_t ) + \mathcal{Q}_{t+1}( x_{t}, D_t  )   \\ 
x_t \in X_t(x_{t-1}, \xi_t ).
\end{array}
\right.
\end{equation}

Finally, let us consider the case when $\xi_t$ does not depend on $(\xi_{[t-1]}, D_t )$
and $D_t$ only depends on $D_{t-1}$. In this setting, $(D_t)$ is an inhomogeneous Markov chain
with two states: an absorbing  state corresponding to the case when the optimization period is over and a
second state where the optimization period is still not over.
We assume that the distribution of $\xi_t$ is discrete with finite support $\{\xi_{t 1}, \ldots,\xi_{t M_t}\}$ with
$p_{t j}=\mathbb{P}(\xi_t = \xi_{t j})$.
The scenario trees for $(\xi_1,\ldots,\xi_{T_{\max}})$
and $((D_1, \xi_1), (D_2, \xi_2), \ldots,(D_{T_{\max}}, \xi_{T_{\max}}))$ (nodes and transition probabilities)
are represented in Figure \ref{figs1} (right and bottom left plots) on an example where $T_{\max}=3$ and where $\xi_t$ has two possible realizations 
for all $t=2,\ldots,T_{\max}$.

With these assumptions  dynamic programming equations
\eqref{dp6}-\eqref{dp7} can be written as follows: $\mathcal{Q}_{T_{\max}+1}(x_{T_{\max}}  ,   D_{T_{\max}}     ) \equiv 0$,
\begin{equation}\label{dp1simp0}
\mathcal{Q}_{t}(x_{t-1}, 0 )=0,\;t=2,\ldots,T_{\max}, 
\end{equation}
and 
\begin{equation}\label{dp1simp}
\mathcal{Q}_t( x_{t-1}, 1 ) = (1-q_t) \sum_{j=1}^{M_{t}} p_{t j} \mathfrak{Q}_t ( x_{t-1}, 1, \xi_{t j}, 1 ) +
q_t \sum_{j=1}^{M_t} p_{t j} \mathfrak{Q}_t ( x_{t-1}, 1, \xi_{t j}, 0 ),
\end{equation}
where $q_t$ is given by \eqref{qtformula},
\begin{equation}\label{dp1simpbis1}
\mathfrak{Q}_t ( x_{t-1}, 1, \xi_{t j}, 1 ) = 
\left\{ 
\begin{array}{l}
\inf_{x_t} \;f_t(x_t, x_{t-1}, \xi_{t j} ) + \mathcal{Q}_{t+1}( x_{t}, 1  )   \\ 
x_t \in X_t(x_{t-1}, \xi_{t j} ),
\end{array}
\right.
\end{equation}
and
\begin{equation}\label{dp1simpbis2}
\mathfrak{Q}_t ( x_{t-1}, 1, \xi_{t j}, 0 ) =   
\left\{ 
\begin{array}{l}
\inf_{x_t} \;f_t(x_t, x_{t-1}, \xi_{t j} )    \\ 
x_t \in X_t(x_{t-1}, \xi_{t j} ).
\end{array}
\right.
\end{equation}
\begin{rem} Observe that the dynamic programming equations above correspond to a model that minimizes the expected cost 
with respect to the distribution of $(\xi_1,\xi_2,\ldots,\xi_{T_{\max}},D_1,D_2,\ldots,D_{T_{\max}})$. From the Law of Large Numbers,
this model is useful  when the corresponding policy is 
repeatedly applied by individuals sharing the same  
distribution of the number of stages $T$. An example would be a group of companies sharing the same distribution for their lifetime, see \cite{lifetimecompanies2015} and Section \ref{sec:numresults}.
\end{rem}

\section{Dynamic programming equations for more general models}\label{dpmoregeneral}

\subsection{Writing dynamic programming equations}\label{moregendp}

In the previous section, we considered models where the cost functions for stages 
$t \leq T$ are taken from the collection of functions $(f_t)$, namely $f_t$
for stage $t$ as long as $t \leq T$. It is possible to write dynamic programming equations for more 
general risk-neutral stochastic programming models having a random number $T$ of stages 
where the cost function for stage $t$ is taken from collection of functions 
$(f_{t,1},f_{t,2},f_{t,3},\ldots,)$ as follows: there is  some random variable $T_1$ such that
for $1 \leq t \leq \min(T_1 , T)$ the cost function is $f_{t,1}$;
next there is some random variable $T_2$ such that
for $\min(T_1 , T)+1 \leq t \leq \min(T_1 +T_2 , T)$ the cost function is $f_{t,2}$, and so on...

As a special case, assume that $T_1 = T-1$: for $t=1,\ldots,T-1$, the cost function
is $f_t(x_t,x_{t-1},\xi_t)$ and for $t=T$ the cost function is ${\overline{f}}_t(x_t,x_{t-1},\xi_t)$.
In this situation, recalling definition \eqref{indicatordef} of $D_t$, the cost function for stage $t$ can be written
\begin{equation}\label{costfunction}
D_{t} f_t(x_t, x_{t-1}, \xi_t ) + ( D_{t-1} - D_t ) {\overline{f}}_t(x_t, x_{t-1}, \xi_t ).
\end{equation}
Indeed, 
\begin{itemize}
\item if $t<T$ we have $D_{t-1}=D_t=1$ and 
$D_{t} f_t(x_t, x_{t-1}, \xi_t ) + ( D_{t-1} - D_t ) {\overline{f}}_t(x_t, x_{t-1}, \xi_t )=f_t(x_t, x_{t-1}, \xi_t )$;
\item if $t=T$ we have $D_{t-1}=1, D_t =0$, and 
$D_{t} f_t(x_t, x_{t-1}, \xi_t ) + ( D_{t-1} - D_t ) {\overline{f}}_t(x_t, x_{t-1}, \xi_t )={\overline f}_t(x_t, x_{t-1}, \xi_t )$;
\item if $t>T$ no costs are incurred, we have $D_{t-1}=D_t=0$ and 
$D_{t} f_t(x_t, x_{t-1}, \xi_t ) + ( D_{t-1} - D_t ) {\overline{f}}_t(x_t, x_{t-1}, \xi_t )=0$;
\end{itemize}
as required.
In the next section we present a simple portfolio problem modelled by a MSP of this type. Therefore, for the ``special case"
we are dealing with now, with cost function \eqref{costfunction} for stage $t$, we obtain the multistage
stochastic program
\begin{equation} \label{reformpb1bis}
\begin{array}{l}
{\inf} \; \mathbb{E}_{\xi_2,\ldots,\xi_{T_{\max}},D_2,\ldots,D_{T_{\max}}}\Big[ \displaystyle{\sum_{t=1}^{T_{\max}}} \; D_{t} f_t(x_t, x_{t-1}, \xi_t ) + ( D_{t-1} - D_t ) {\overline{f}}_t(x_t, x_{t-1}, \xi_t )\Big]\\
x_t \in X_t( x_{t-1}, \xi_t )\;\mbox{a.s.}, x_{t} \;{\overline{\mathcal{F}}}_t\mbox{-measurable, }t=1,\ldots,T_{\max},
\end{array}
\end{equation}
where  ${\overline{\mathcal{F}}}_  t$ is the sigma-algebra given by \eqref{deffttr}.
 Observe that when $${\overline{f}}_t(x_t, x_{t-1}, \xi_t )=f_t(x_t, x_{t-1}, \xi_t ),$$ we are back to 
the stochastic programs considered in the previous section, i.e., problem  
\eqref{reformpb1bis} becomes problem \eqref{reformpb1}. 
Clearly, \eqref{reformpb1bis} is obtained replacing in \eqref{reformpb1} the cost function $D_{t} f_t(x_t, x_{t-1}, \xi_t )$
for stage $t$ by \eqref{costfunction}.
Therefore dynamic programming equations for \eqref{reformpb1bis} are obtained updating correspondingly  the cost functions in 
dynamic programming equations \eqref{dp4}-\eqref{dp5} for \eqref{reformpb1}, i.e., the dynamic programming equations for \eqref{reformpb1bis} are: 
$\mathcal{Q}_{T_{\max}+1}(x_{T_{\max}}  , \xi_{[T_{\max}]},  D_{T_{\max}}  ) \equiv 0$ and for $t=2,\ldots,T_{\max}$,
\begin{equation}\label{dp4bis}
\mathcal{Q}_{t}(x_{t-1}, \xi_{[t-1]}, D_{t-1} )=\mathbb{E}_{\xi_t, D_t } \Big[ \mathfrak{Q}_t (  x_{t-1}, \xi_{[t-1]}, D_{t-1} , \xi_t    , D_t ) |  D_{t-1} , \xi_{[t-1]}   \Big]
\end{equation}
where 
{\small{
\begin{equation}\label{dp5bis}
\mathfrak{Q}_t(  x_{t-1}, \xi_{[t-1]}, D_{t-1} , \xi_t    , D_{t}     ) = 
\left\{ 
\begin{array}{l}
\inf_{x_t} \;D_{t} f_t(x_t, x_{t-1}, \xi_t ) + ( D_{t-1} - D_t ) {\overline{f}}_t(x_t, x_{t-1}, \xi_t )  + \mathcal{Q}_{t+1}( x_{t}, \xi_{[t]}, D_t  )   \\ 
x_t \in X_t(x_{t-1}, \xi_t ).
\end{array}
\right.
\end{equation}
}}
Now assume that $\xi_t$ does not depend on $(\xi_{[t-1]}, D_t )$, $D_t$ only depends on $D_{t-1}$,
and the distribution of $\xi_t$ is discrete with finite support $\{\xi_{t 1}\ldots,\xi_{t M_t}\}$.
Denoting $p_{t j}=\mathbb{P}(\xi_t = \xi_{t j})$, equations \eqref{dp4bis}-\eqref{dp5bis} simplify
as follows: $\mathcal{Q}_{T_{\max}+1}(x_{T_{\max}}, 0 )=\mathcal{Q}_{T_{\max}+1}(x_{T_{\max}}, 1 )  \equiv 0$, 
for $t=2,\ldots,T_{\max}$, $\mathcal{Q}_{t}(x_{t-1}, 0 ) \equiv 0$, and  for $t=2,\ldots,T_{\max},$ we have
\begin{equation}\label{dp1simpter}
\mathcal{Q}_t( x_{t-1}, 1 ) = (1-q_t) \sum_{j=1}^{M_{t}} p_{t j} \mathfrak{Q}_t ( x_{t-1}, 1, \xi_{t j}, 1 ) +
q_t \sum_{j=1}^{M_t} p_{t j} \mathfrak{Q}_t ( x_{t-1}, 1, \xi_{t j}, 0 ),
\end{equation}
where $q_t$ is given by \eqref{qtformula},
\begin{equation}\label{dp7ter}
\mathfrak{Q}_t(  x_{t-1}, 1 , \xi_{t j}    , 1     ) = 
\left\{ 
\begin{array}{l}
\inf_{x_t} \;f_t(x_t, x_{t-1}, \xi_{t j} ) + \mathcal{Q}_{t+1}( x_{t}, 1  )    \\ 
x_t \in X_t(x_{t-1}, \xi_{t j} ),
\end{array}
\right.
\end{equation}
and
\begin{equation}\label{dp7ter}
\mathfrak{Q}_t(  x_{t-1}, 1 , \xi_{t j}    , 0     ) = 
\left\{ 
\begin{array}{l}
\inf_{x_t} \; {\overline{f}}_t(x_t, x_{t-1}, \xi_{t j} )    \\ 
x_t \in X_t(x_{t-1}, \xi_{t j} ).
\end{array}
\right.
\end{equation}

\subsection{Example: a simple portfolio problem} \label{simpleportpb}

We consider the portfolio selection problem with direct transaction costs given in \cite{guilejtekregsddp}. 
When the number of stages is random we obtain a problem from the class of problems introduced in the previous Section \ref{moregendp}.
We first recall the dynamic programming equations for this model when the number of stages $T_{\max}$ is fixed and known.
We refer to Table \ref{tableparamdbv} for the list of corresponding parameters and decision variables.
\begin{table}
\begin{center}
\begin{tabular}{|l|c|c|}
\hline
\multirow{4}{*}{Decision variables} & $x_t(i)$ & Dollar value of risky asset $i$ at end of $t$  \\
\cline{2-3}
& $y_t(i)$ & Amount (in dollars) of risky asset $i$ sold at end of $t$\\
\cline{2-3}
& $z_t(i)$ & Amount (in dollars) of risky asset $i$ bought at end of $t$ \\
\cline{2-3}
& $x_t(n+1)$ & Amount (in dollars) held in cash at end of $t$\\
\hline
Stochastic parameters & $\xi_t(i)$ & Return of risky asset $i$ for stage $t$\\
\hline
\multirow{5}{*}{Deterministic parameters} & $\xi_t(n+1)$ & Cash return for stage $t$  \\
\cline{2-3}
&$u(i)$& Maximal proportion of capital invested in risky asset $i$ \\
\cline{2-3}
&$\eta_t(i)$& Selling transaction cost for risky asset $i$ stage $t$ \\
\cline{2-3}
&$\nu_t(i)$& Buying transaction cost for risky asset $i$ stage $t$ \\
\cline{2-3}
&$x_0$& Initial portfolio  \\
\hline
\end{tabular}
\end{center}
\caption{Parameters and decision variables for the portfolio problem.}
\label{tableparamdbv}
\end{table}
Let $x_t( i )$ be the dollar value of asset $i=1,\ldots,n+1$, at the end of stage $t=1,\ldots,T_{\max}$,
where asset $n+1$ is cash; let $\xi_t(i)$ be the return of asset $i$ at $t$; 
let $y_t(i)$ be the amount of asset $i$ sold at the end of $t$; 
let $z_t(i)$ be the amount of asset $i$ bought at the end of $t$ with
$\eta_t(i) > 0$ and $\nu_t(i) > 0$ the respective proportional selling and buying transaction costs at $t$.
Each component $x_0(i),i=1,\ldots,n+1$, of $x_0$ is known.
The budget available at the beginning of the investment period is
$\sum_{i=1}^{n+1} \xi_{1}(i) x_{0}(i)$ and
$u(i)$ represents the maximal proportion of capital that can be invested in asset $i$. 
For $t=1,\ldots,T_{\max}$, given a portfolio $x_{t-1}=(x_{t-1}(1),\ldots,x_{t-1}(n), x_{t-1}({n+1}))$ and $\xi_t$, we define the set $X_t(x_{t-1}, \xi_t)$
as the set of portfolios $(x_t, y_t, z_t) \in \mathbb{R}^{n+1} \small{\times} \mathbb{R}^{n} \small{\times} \mathbb{R}^{n}$ satisfying
$$
\begin{array}{l}
x_t( n+1 ) = \xi_{t}( n+1 ) x_{t-1}( n+1 )   +\sum\limits_{i=1}^{n} \Big((1-\eta_t( i) )y_t( i )-  (1+\nu_t ( i ))z_t( i )\Big),\\
x_{t}(i)= \xi_{t}(i) x_{t-1}(i)-y_t(i) +z_t(i), i=1,\ldots,n,\\
x_t(i) \leq u(i)  \sum\limits_{j=1}^{n+1} \xi_{t}(j) x_{t-1}(j),i=1,\ldots,n,\\
x_t(i) \geq 0, y_t(i) \geq 0, z_t(i) \geq 0,i=1,\ldots,n.
\end{array}
$$
With this notation, the following dynamic programming equations of a risk-neutral portfolio model can be written:
for $t=T_{\max}$, we solve the problem
\begin{equation}\label{eq1dp}
\mathfrak{Q}_{T_{\max}} \left( x_{T_{\max}-1}, \xi_{T_{\max}} \right)=
\left\{
\begin{array}{l}
\inf \; {\overline{f}}_{T_{\max}}(x_{T_{\max}}, x_{T_{\max}-1}, \xi_{T_{\max}} ) :=   -  \mathbb{E}\Big[ \sum\limits_{i=1}^{n+1} \xi_{T_{\max}+1}(i) x_{T_{\max}}(i) \Big]  \\
X_{T_{\max}} \in X_{T_{\max}}(x_{T_{\max}-1}, \xi_{T_{\max}}),
\end{array}
\right.
\end{equation}
while at stage $t=T_{\max}-1,\dots,1$, we solve
\begin{equation}\label{eq2dp} 
\mathfrak{Q}_t\left( x_{t-1}, \xi_t  \right)=
\left\{
\begin{array}{l}
\inf  \;  Q_{t+1}\left( x_{t} \right) \\
x_t \in X_t( x_{t-1} , \xi_{t}),
\end{array}
\right.
\end{equation}
where
\begin{equation}\label{eq3dpok}
\mathcal{Q}_t(x_{t-1})=\mathbb{E}_{\xi_t}[ \mathfrak{Q}_t\left( x_{t-1}, \xi_t  \right) ],\;t=2,\ldots,T_{\max}.
\end{equation}
Now for $t=1,\ldots,T_{\max}$, define 
\begin{equation}\label{fandfbar}
f_t(x_t,x_{t-1},\xi_t) \equiv 0 \mbox{ and }
{\overline{f}}_{t} (x_{t}, x_{t-1}, \xi_{t}) = -\mathbb{E}\Big[ \sum \limits_{i=1}^{n+1} \xi_{t+1}(i) x_{t}(i) \Big].
\end{equation}
Since the number of stages is fixed to $T_{\max}$ then $D_t=1,t=1,\ldots,T_{\max}-1, D_{T_{\max}}=0$ almost surely, and the porfolio
problem we have just described is of form \eqref{reformpb1bis} with $f_t, {\overline{f}}_t$ as in \eqref{fandfbar} and
$D_t=1,t=1,\ldots,T_{\max}-1, D_{T_{\max}}=0$ almost surely, i.e., we obtain the
portfolio problem
\begin{equation} \label{reformpb1bister}
\begin{array}{l}
{\inf} \; \mathbb{E}_{\xi_2,\ldots,\xi_{T_{\max}}}\Big[ {\overline{f}}_{T_{\max}} (x_{T_{\max}}, x_{T_{\max}-1}, \xi_{T_{\max}}) \Big]\\
x_t \in X_t( x_{t-1}, \xi_t )\;\mbox{a.s.}, x_{t} \;\mathcal{F}_t\mbox{-measurable, }t=1,\ldots,T_{\max}.
\end{array}
\end{equation}
With this model, we minimize the expected loss of the portfolio (or equivalently maximize the mean income) taking into account the transaction costs,
non-negativity constraints, and bounds imposed on the different securities.

Now assume that  the number of stages is random with discrete distribution on $\{2,\ldots,T_{\max}\}$
and define $D_t$ by \eqref{indicatordef}. We obtain the portfolio problem \eqref{reformpb1bis} with $f_t, {\overline{f}}_t$ as in \eqref{fandfbar}.  
If $\xi_t$ does not depend on $(\xi_{[t-1]}, D_t )$, $D_t$ only depends on $D_{t-1}$,
and the distribution of $\xi_t$ is discrete with finite support $\{\xi_{t 1}\ldots,\xi_{t M_t}\}$,
denoting $p_{t j}=\mathbb{P}(\xi_t = \xi_{t j})$, we can write the following dynamic programming equations for the 
corresponding portfolio problem: $\mathcal{Q}_{T_{\max}+1}(x_{T_{\max}}, 0 )=\mathcal{Q}_{T_{\max}+1}(x_{T_{\max}}, 1 )  \equiv 0$, 
for $t=2,\ldots,T_{\max}$, $\mathcal{Q}_{t}(x_{t-1}, 0 ) \equiv 0$ and  for $t=2,\ldots,T_{\max},$ we have
\begin{equation}\label{dp1simpterport}
\mathcal{Q}_t( x_{t-1}, 1 ) = (1-q_t) \sum_{j=1}^{M_{t}} p_{t j} \mathfrak{Q}_t ( x_{t-1}, 1, \xi_{t j}, 1 ) +
q_t \sum_{j=1}^{M_t} p_{t j} \mathfrak{Q}_t ( x_{t-1}, 1, \xi_{t j}, 0 ),
\end{equation}
where $q_t$ is given by \eqref{qtformula},
\begin{equation}\label{dp7terport}
\mathfrak{Q}_t(  x_{t-1}, 1 , \xi_{t j}    , 1     ) = 
\left\{ 
\begin{array}{l}
\inf_{x_t} \mathcal{Q}_{t+1}( x_{t}, 1  )    \\ 
x_t \in X_t(x_{t-1}, \xi_{t j} ),
\end{array}
\right.
\end{equation}
and
\begin{equation}\label{dp7terport2}
\mathfrak{Q}_t(  x_{t-1}, 1 , \xi_{t j}    , 0     ) = 
\left\{ 
\begin{array}{l}
\inf_{x_t} \;  - \mathbb{E}[\xi_{t+1}^T x_t  ] \\ 
x_t \in X_t(x_{t-1}, \xi_{t j} ).
\end{array}
\right.
\end{equation}
In the case when the number of stages is $T_{\max}$ (deterministic) then
$q_t=0$ for $t=2,\ldots,T_{\max}-1$, $q_{T_{\max}}=1$ and dynamic programming equations 
\eqref{dp1simpterport}, \eqref{dp7terport}, \eqref{dp7terport2} become, as expected,
\eqref{eq1dp}, \eqref{eq2dp}, \eqref{eq3dpok} (with the notation $\mathcal{Q}_t(x_{t-1})$ instead
of $\mathcal{Q}_t(x_{t-1}, 1$), $\mathfrak{Q}_t(x_{t-1}, \xi_{t j})$ instead of $\mathfrak{Q}_t(x_{t-1}, 1,  \xi_{t j} , 1)$,
and $\mathfrak{Q}_{T_{\max}} (x_{T_{\max}-1}, \xi_{T_{\max}})$ instead of $\mathfrak{Q}_{T_{\max}}(x_{T_{\max}-1}, 1,  \xi_{T_{\max}} , 0)$).

\section{SDDP for multistage stochastic risk-neutral programs with a random number of stages}\label{sec:sddp}

\subsection{Assumptions}

Consider optimization problem \eqref{reformpb1} where $\xi_t$ does not depend on $(\xi_{[t-1]}, D_t )$
and $D_t$ only depends on $D_{t-1}$. We assume that the distributions of $T$ and $\xi_t$ are discrete:
the support of $T$ is
$\{2,\ldots,T_{\max}\}$ and the support of  $\xi_t$ is $\Theta_t = \{\xi_{t 1},\ldots,\xi_{t M_t} \}$ with
$p_{t i}= \mathbb{P}(\xi_t = \xi_{t i}) >0, i=1,\ldots,M_t$.
In this context, equations \eqref{dp1simp0}, \eqref{dp1simp}, \eqref{dp1simpbis1}, \eqref{dp1simpbis2} are the dynamic programming equations for \eqref{reformpb1}.
We can now apply Stochastic Dual Dynamic Programming (SDDP, \cite{pereira}), 
to solve these dynamic programming equations
as long as recourse functions $\mathcal{Q}_{t}(\cdot,1)$ are convex.
SDDP has been used to solve many real-life problems and several extensions of the method have been considered 
such as DOASA \cite{philpot}, CUPPS \cite{chenpowell99}, ReSA \cite{resa}, AND \cite{birgedono},
risk-averse (\cite{guiguesrom10, guiguesrom12, kozmikmorton, philpmatos, shapsddp, shadenrbook}) or inexact (\cite{guigues2016isddp}) variants;
see also  \cite{guiguescoap2013, morton} for adaptations to interstage dependent processes and \cite{shabbirzou} for extensions for
integer stochastic programs.

SDDP builds approximations for the cost-to-go functions which take the form of a maximum of
affine functions. 

To ensure convexity of functions $\mathcal{Q}_{t}(\cdot,1)$, we need convexity of functions 
$f_t(\cdot,\cdot,\xi_t)$ and of multifunctions
$X_t(\cdot,\xi_t)$  for almost every $\xi_t$. We will consider two settings: linear and nonlinear programs.\\

\par {\textbf{Linear problems.}} In this setting, $f_t(x_t,x_{t-1},\xi_t)=c_t^T x_t$ is linear,
\begin{equation}\label{xtlin}
X_t(x_{t-1}, \xi_t):=\{x_t \in \mathbb{R}^n : A_{t} x_{t} + B_{t} x_{t-1} = b_t, \,x_t \geq 0 \},
\end{equation}
and random vector $\xi_t$ corresponds to the concatenation of the elements in random matrices $A_t, B_t$ which have a known
finite number of rows and random vectors $b_t, c_t$. We assume:
\begin{itemize} 
\item[(H2-L)] The set $X_1(x_{0}, \xi_1 )$ is nonempty and bounded and for every $x_1 \in X_1(x_{0}, \xi_1 )$,
for every $t=2,\ldots,T$, for every realization $\tilde \xi_2, \ldots, \tilde\xi_t$ of $\xi_2,\ldots,\xi_t$,
for every $x_{\tau} \in X_{\tau}( x_{\tau-1} , \tilde \xi_{\tau}), \tau=2,\ldots,t-1$, the set $X_t( x_{t-1} , {\tilde \xi}_t )$
is nonempty and bounded.\\
\end{itemize}

\par {\textbf{Nonlinear problems.}} In this case,
\begin{equation}\label{xtnonlin}
X_t( x_{t-1} , \xi_t)= \{x_t \in \mathbb{R}^n : x_t \in \mathcal{X}_t,\;g_t(x_t, x_{t-1}, \xi_t) \leq 0,\;\;\displaystyle A_{t} x_{t} + B_{t} x_{t-1} = b_t \},
\end{equation}
and $\xi_t$ contains in particular the random elements in matrices $A_t, B_t$, and vector $b_t$. 
Of course, as a special case (and as is often the case in applications), the nonlinear problems we are interested in can have nonlinear cost and constraint functions for stage $t$ that do not depend on
$x_{t-1}$, namely of form $f_t(x_t, \xi_t)$ and $g_t(x_t, \xi_t)$.\\

We assume
that for $t=1,\ldots,T$, there exists $\varepsilon_t>0$ such that:\\
\par (H2-NL)-(a) $\mathcal{X}_t$ is nonempty, convex, and compact.
\par (H2-NL)-(b) For every $j=1,\ldots,M_t$, the function $f_t(\cdot, \cdot,\xi_{t j})$ is convex, lower semicontinuous, and finite on  $\mathcal{X}_t \small{\times} \mathcal{X}_{t-1}^{\varepsilon_t}$
where $\mathcal{X}_{t-1}^{\varepsilon_t}=\mathcal{X}_{t-1} + \{x \in \mathbb{R}^n : \|x\|_2 \leq \varepsilon_t \}$.
\par (H2-NL)-(c) for every $j=1,\ldots,M_t$, each component $g_{t, i}(\cdot, \cdot, \xi_{t j}), i=1,\ldots,p$, of the function \par$g_t(\cdot, \cdot, \xi_{t j})$ is 
convex, lower semicontinuous, and finite  on $\mathcal{X}_t \small{\times} \mathcal{X}_{t-1}^{\varepsilon_t}$.
\par (H2-NL)-(d) For every $j=1,\ldots,M_t$, for every $x_{t-1} \in \mathcal{X}_{t-1}^{\varepsilon_t}$, the set 
$X_t( x_{t-1}, \xi_{t j})$ is nonempty.
\par (H2-NL)-(e) If $t \geq 2$, for every $j=1,\ldots,M_t$, there exists
$${\bar x}_{t, j}=( {\bar x}_{t, j, t} , {\bar x}_{t, j, t-1} ) \in \mbox{ri}(\mathcal{X}_t) \small{\times}  \mathcal{X}_{t-1}
\cap \mbox{ri}(\{g_t(\cdot,\cdot,\xi_{t j})\leq 0\})$$ 
\par such that $\bar x_{t, j, t} \in X_t(\bar x_{t, j, t-1}, \xi_{t j})$.\\

Assumptions (H2-NL)-(a),(b),(c) in the nonlinear case imply the convexity of cost-to-go functions
$\mathcal{Q}_t(\cdot,1)$. The assumptions above also ensure (both in the linear and nonliner cases) that SDDP applied to dynamic programming equations \eqref{dp1simp}, \eqref{dp1simpbis1}, \eqref{dp1simpbis2} 
will converge, as long as samples in the forward passes are independent, see \cite{philpot, guiguessiopt2016, lecphilgirar12}  for details.

\subsection{Algorithm}\label{sec:algosddptsto}

We now describe the steps of SDDP applied to dynamic programming equations \eqref{dp1simp0}, \eqref{dp1simp}, \eqref{dp1simpbis1}, \eqref{dp1simpbis2}. We denote by SDDP-TSto this SDDP method
for solving \eqref{reformpb1}\footnote{TSto in acronym SDDP-TSto refers to the fact that this SDDP method solves stochastic programs with $T$ stochastic, $T$
being the number of stages.}. 
SDDP-TSto is similar to the variant of SDDP presented in \cite{philpmatos, mogrundt}
where the underlying stochastic process depends on a Markov Chain.\footnote{However, in \cite{philpmatos, mogrundt}, problems are linear whereas
we will detail SDDP-TSto both for linear and nonlinear stochastic programs.} This Markov chain (process $(D_t)$ for our DP equations) has only two states in our case.
The cost-to-go function is null in one of these states (when $D_{t-1}=0$) and the goal of SDDP-TSto is to approximate the cost-to-go function in the other state (when $D_{t-1}=1$) for all stages, i.e., cost-to-go functions
$\mathcal{Q}_{t}(\cdot,1),t=2,\ldots,T_{\max}$.

In the end of iteration $k$, the algorithm has computed
for 
cost-to-go functions $\mathcal{Q}_{t}(\cdot,1),t=2,\ldots,T_{\max}$,
the approximations $\mathcal{Q}_t^k(\cdot,1),t=2,\ldots,T_{\max}$, which are maximum of  $k+1$ affine functions called cuts:
\begin{equation}\label{sddpapprox}
\mathcal{Q}_t^k(x_{t-1},1) = \displaystyle \max_{0 \leq j \leq k} \;\theta_t^j + \langle \beta_t^j , x_{t-1} \rangle.
\end{equation}
At iteration $k$, a realization of the number of stages and a sample
for $(\xi_t)_{1 \leq t \leq T_{\max}}$, are generated.
Decisions $x_{t}^k, t=1,\ldots,T_{\max}$, are computed on this sample in a forward pass replacing (unknown) function 
$\mathcal{Q}_t(x_{t-1},1)$ by  $\mathcal{Q}_t^{k-1}(x_{t-1},1)$.
In the backward pass of iteration $k$, decisions $x_t^k$ are then used to compute coefficients
$\theta_t^k, \beta_t^k, t=2,\ldots,T_{\max}$.\\

\par {\textbf{SDDP-TSto, Step 1: Initialization.}} Fix an integer $N>1$. For $t=2,\ldots,T_{\max}$, take for $\mathcal{Q}_t^0(\cdot, 1)$ 
a known lower bounding affine function $\theta_t^0 + \langle \beta_t^0 , \cdot \rangle$  for $\mathcal{Q}_t(\cdot,1)$.
Set the iteration count $k$ to 1 and $\mathcal{Q}_{T_{\max}+1}^0(\cdot,1)=\mathcal{Q}_{T_{\max}+1}^0(\cdot,0)\equiv 0$,
$\mathcal{Q}_{t}^0(\cdot,0) \equiv 0, t=2,\ldots,T_{\max}$. Fix a parameter $0<\mbox{Tol}<1$ (for the stopping criterion).
Compute $q_t,t=2,\ldots,T_{\max}$ using  \eqref{qtformula} starting from $q_2=\mathbb{P}(T=2)$.\\
\par {\textbf{SDDP-TSto, Step 2: Forward pass.}}
We generate a sample
$$ 
(  (\tilde \xi_1^k ,  \tilde D_1^k ),(\tilde \xi_2^k ,  \tilde D_2^k ),\ldots,(\tilde \xi_{T_{\max}}^k ,  \tilde D_{T_{\max}}^k )),
$$
from the distribution of 
$$
\gamma^k =  ((\xi_1^k, D_1^k ),(\xi_2^k, D_2^k),\ldots,(\xi_{T_{\max}}^k, D_{T_{\max}}^k )) \sim
((\xi_1, D_1 ),(\xi_2, D_2),\ldots,(\xi_{T_{\max}}, D_{T_{\max}} )),
$$
with the convention that $\tilde \xi_1^k = \xi_1$, $\tilde D_1^k=1$. \\
Cost$_k =$ 0. $t \leftarrow 1$.\\
{\textbf{While }}$\tilde D_t^k=1$, we compute an optimal solution $x_t^k$ of 
\begin{equation}\label{pbforwardpassb}
\left\{
\begin{array}{l}
\inf_{x_t \in \mathbb{R}^n} f_t(x_t , x_{t-1}^k, \tilde \xi_t^k  ) + \mathcal{Q}_{t+1}^{k-1} ( x_t , 1  )\\
x_t \in X_t(x_{t-1}^k, {\tilde \xi}_{t}^k )
\end{array}
\right.
\end{equation}
\hspace*{1.2cm}where $x_0^k=x_0$.\\
\hspace*{1.2cm}Cost$_k \leftarrow \mbox{Cost}_k + f_t(x_t^k , x_{t-1}^k, \tilde \xi_t^k  )$.\\
\hspace*{1.2cm}$t \leftarrow t+1$.\\
{\textbf{End While}}\\
We compute an optimal solution $x_t^k$ of 
\begin{equation}\label{pbforwardpass1}
\left\{
\begin{array}{l}
\inf_{x_t \in \mathbb{R}^n} f_t(x_t , x_{t-1}^k, \tilde \xi_t^k  )   \\
x_t \in X_t(x_{t-1}^k, {\tilde \xi}_{t}^k ).
\end{array}
\right.
\end{equation}
Cost$_k \leftarrow \mbox{Cost}_k + f_t(x_t^k , x_{t-1}^k, \tilde \xi_t^k  )$.\\
$t \leftarrow t+1$.\\
{\textbf{While }}($t\leq T_{\max}-1)$, we compute an optimal solution $x_t^k$ of 
\begin{equation}\label{pbforwardpassbis}
\left\{
\begin{array}{l}
\inf_{x_t \in \mathbb{R}^n} 0\\
x_t \in X_t(x_{t-1}^k, {\tilde \xi}_{t}^k )
\end{array}
\right.
\end{equation}
\hspace*{1.2cm}(note that $\tilde D_t^k=0$ and the objective function is null, we only need a feasible point).\\
\hspace*{1.2cm}$t \leftarrow t+1$.\\
{\textbf{End While}}\\
\par {\textbf{Upper bound computation:}} If $k \geq N$ compute
$$
{\overline{\mbox{Cost}}}_k = \frac{1}{N} \sum_{j=k-N+1}^k \mbox{Cost}_j,\;\;{\hat \sigma}_{N,k}^2= \frac{1}{N}\sum_{j=k-N+1}^k [\mbox{Cost}_j - {\overline{\mbox{Cost}}}_k ]^2
$$
and the upper bound
$$
{\overline U}_k = {\overline{\mbox{Cost}}}_k + \frac{{\hat \sigma}_{N,k}}{\sqrt{N}} t_{N-1, 1-\alpha}
$$
where $t_{N-1, 1-\alpha}$ is the $(1-\alpha)$-quantile of the Student distribution with $N-1$ degrees of freedom.\\
\par {\textbf{SDDP-TSto, Step 3: Backward pass.}} Let ${\underline{\mathfrak{Q}}}_t^k( x_{t-1} , D_{t-1} , \xi_t, D_t )$ be the function given by
\begin{equation}\label{dp7k}
{\underline{\mathfrak{Q}}}_t^k (  x_{t-1}, D_{t-1} , \xi_t    , D_{t}     ) = 
\left\{ 
\begin{array}{l}
\inf_{x_t} \;D_{t-1} f_t(x_t, x_{t-1}, \xi_t ) + \mathcal{Q}_{t+1}^k ( x_{t}, D_t  )   \\ 
x_t \in X_t(x_{t-1}, \xi_t ).
\end{array}
\right.
\end{equation}
Set $\mathcal{Q}_{T_{\max}+1}^{k}(\cdot,1) =  \mathcal{Q}_{T_{\max}+1}^{k}(\cdot,0) \equiv 0$.\\
{\textbf{For}} $t=T_{\max}$ down to $t=2$, \\
\hspace*{0.5cm}Set $\mathcal{Q}_{t}^{k}(\cdot,0) \equiv 0$.\\
\hspace*{0.5cm}{\textbf{For }}$j=1,\ldots,M_t$,\\
\hspace*{1.2cm}Compute ${\underline{\mathfrak{Q}}}_t^k ( x_{t-1}^k, 1, \xi_{t j}, 1 )$, compute 
$$
\mathfrak{Q}_t  ( x_{t-1}^k, 1, \xi_{t j}, 0 )=
\left\{ 
\begin{array}{l}
\inf_{x_t} \;f_t(x_t, x_{t-1}^k, \xi_{t j} )  \\ 
x_t \in X_t(x_{t-1}^k, \xi_{t j} ),
\end{array}
\right.
$$
\hspace*{1.2cm}compute a subgradient $\beta_{t j}^k$ of  ${\underline{\mathfrak{Q}}}_t^k ( \cdot, 1, \xi_{t j}, 1 )$ at $x_{t-1}^k$
and a subgradient $\gamma_{t j}^k$ of \\
\hspace*{1.2cm}$\mathfrak{Q}_t  ( \cdot , 1, \xi_{t j}, 0 )$ at $x_{t-1}^k$.\\
\hspace*{0.5cm}{\textbf{End For}}\\
\hspace*{0.5cm}Compute
\begin{equation}\label{formulathetakbetak}
\begin{array}{lll}
\theta_t^k &= &(1-q_t)  \displaystyle  \sum_{j=1}^{M_t} p_{t j} \Big( {\underline{\mathfrak{Q}}}_t^k ( x_{t-1}^k, 1, \xi_{t j}, 1 )   -  \langle \beta_{t j}^k , x_{t-1}^k  \rangle   \Big) \\
&&+ q_t  \displaystyle \sum_{j=1}^{M_t} p_{t j} \Big( \mathfrak{Q}_t ( x_{t-1}^k, 1, \xi_{t j}, 0 )   -  \langle \gamma_{t j}^k , x_{t-1}^k  \rangle   \Big),\\
\beta_t^k &=& (1-q_t) \displaystyle  \sum_{j=1}^{M_t} p_{t j} \beta_{t j}^k + q_t  \sum_{j=1}^{M_t} p_{t j} \gamma_{t j}^k.
\end{array}
\end{equation}
{\textbf{End For}}\\
\par {\textbf{Lower bound computation:}} compute the lower bound $\underline{L}_k$ on the optimal value of \eqref{reformpb1} given by
$$
\underline{L}_k =
\left\{ 
\begin{array}{l}
\inf_{x_1} \;f_1(x_1, x_{0}, \xi_1 ) + \mathcal{Q}_{2}^k ( x_{1}, 1  )   \\ 
x_1 \in X_1(x_{0}, \xi_1 ).
\end{array}
\right.
$$
\par {\textbf{SDDP-TSto, Step 4:} If $k \geq N$ and  $\frac{{\overline U}_k - \underline{L}_k}{{\overline U}_k} \leq \mbox{Tol}$ then stop otherwise do $k \leftarrow k+1$ and go to Step 2.\\

We now show that the cuts computed by SDDP-TSto are valid, that $\underline{L}_k$ is a lower bound on the optimal value of
the problem and that the sequence of  approximate first stage problems optimal values converges almost surely
to the optimal value of \eqref{reformpb1}.
\begin{thm} 
Consider optimization problem \eqref{reformpb1} where for $t=1,\ldots,T_{\max}$,
$\xi_t$ does not depend on $(\xi_{[t-1]}, D_t )$
and $D_t$ only depends on $D_{t-1}$. Assume that (H1) holds and that the distribution of $\xi_t$ is discrete for $t=1,\ldots,T_{\max}$.
In the case of linear problems ($X_t$ as in \eqref{xtlin}) assume that (H2-L) holds and in the
case of nonlinear problems ($X_t$ as in \eqref{xtnonlin}) assume that (H2-NL)-(a)-(e) holds.
Consider the sequences $(x_t^k)_{k \geq 1}, t=1,\ldots,T_{\max}$ and $(\mathcal{Q}_t^k(\cdot,1))_{k \geq 0}, t=2,\ldots,T_{\max}$,
generated by SDDP-TSto to solve the corresponding dynamic programming equations \eqref{dp1simp0}, \eqref{dp1simp}, \eqref{dp1simpbis1}, \eqref{dp1simpbis2}.

Assume that samples in the forward passes are independent: the sample
$$
(  (\tilde \xi_1^k ,  \tilde D_1^k ),(\tilde \xi_2^k ,  \tilde D_2^k ),\ldots,(\tilde \xi_{T_{\max}}^k ,  \tilde D_{T_{\max}}^k ))
$$
in the forward pass of iteration $k$ is a realization of random vector 
$$
\gamma^k =  ((\xi_1^k, D_1^k ),(\xi_2^k, D_2^k),\ldots,(\xi_{T_{\max}}^k, D_{T_{\max}}^k )) 
$$
which has the distribution of $((\xi_1, D_1 ),(\xi_2, D_2),\ldots,(\xi_{T_{\max}}, D_{T_{\max}} )) $
and $\gamma^1,\gamma^2,\ldots$ are independent.
Then
\begin{itemize}
\item[(i)] for $t=2,\ldots,T_{\max}+1$, for all $k \geq 0$, $\mathcal{Q}_t^k(\cdot,1)$ is a lower bounding function for 
$\mathcal{Q}_t(\cdot,1)$: for all $x_{t-1}$ we have $\mathcal{Q}_t(x_{t-1},1) \geq \mathcal{Q}_t^k (x_{t-1},1)$ almost surely.
\item[(ii)] $\underline{L}_k$ computed in Step 3 of SDDP-TSto is a lower bound on the optimal value of \eqref{reformpb1}.  
\item[(iii)] Almost surely the limit of the sequence
$( f_1(x_1^k , x_0 , \xi_1 ) + \mathcal{Q}_2^k( x_1^k , 1 ) )_{k \geq 1}$ 
is the optimal value of \eqref{reformpb1}.
\end{itemize}
\end{thm}
\begin{proof}
(i) The proof is by induction on $k$ and $t$. For $k=0$, we have $\mathcal{Q}_t(\cdot,1) \geq \mathcal{Q}_t^0(\cdot,1),t=2,\ldots,T_{\max}+1$.
Now assume that 
\begin{equation}\label{induction}
\mathcal{Q}_t(\cdot,1) \geq \mathcal{Q}_t^{k-1}(\cdot,1),t=2,\ldots,T_{\max}+1,
\end{equation}
for some $k \geq 1$.
We show by backward induction on $t$ that $\mathcal{Q}_t(\cdot,1) \geq \mathcal{Q}_t^{k}(\cdot,1),t=2,\ldots,T_{\max}+1$.
For $t=T_{\max}+1$ we have $\mathcal{Q}_t(\cdot,1) = \mathcal{Q}_t^{k}$ (both functions are null).
Now assume that 
$\mathcal{Q}_{t+1}(\cdot,1) \geq  \mathcal{Q}_{t+1}^{k}(\cdot,1)$
for some $t \in \{2,\ldots,T_{\max}\}$ (induction hypothesis).
We want to show that 
\begin{equation}\label{wanttoshow}
\mathcal{Q}_{t}(\cdot,1) \geq  \mathcal{Q}_{t}^k(\cdot, 1). 
\end{equation}
The induction hypothesis, together with the definitions of ${\underline{\mathfrak{Q}}}_t$
and $\mathfrak{Q}_t^k$ imply that for all $j=1,\ldots,M_t$:
\begin{equation}\label{firstineq}
\mathfrak{Q}_t (  \cdot, 1 , \xi_{t j}    , 1     ) \geq {\underline{\mathfrak{Q}}}_t^k (  \cdot, 1, \xi_{t j}, 1     ). 
\end{equation}
Therefore, we get
\begin{equation}\label{lboundthetabeta}
\begin{array}{lcl}
\mathcal{Q}_t( \cdot, 1 )& \stackrel{\eqref{dp1simp}}{=}    & (1-q_t) \displaystyle  \sum_{j=1}^{M_{t}} p_{t j} \mathfrak{Q}_t ( \cdot, 1, \xi_{t j}, 1 ) +
q_t  \sum_{j=1}^{M_t} p_{t j} \mathfrak{Q}_t ( \cdot, 1, \xi_{t j}, 0 ),\\
& \stackrel{\eqref{firstineq}}{\geq} & (1-q_t) \displaystyle \sum_{j=1}^{M_{t}} p_{t j} {\underline{\mathfrak{Q}}}_t^k (  \cdot, 1, \xi_{t j}, 1     )  +
q_t \sum_{j=1}^{M_t} p_{t j} \mathfrak{Q}_t ( \cdot, 1, \xi_{t j}, 0 ),\\
& \geq & (1-q_t) \displaystyle \sum_{j=1}^{M_{t}} p_{t j} \Big[   {\underline{\mathfrak{Q}}}_t^k (  x_{t-1}^k, 1, \xi_{t j}, 1     )  + \langle \beta_{t j}^k, \cdot - x_{t-1}^k \rangle \Big]\\
& & + q_t \displaystyle \sum_{j=1}^{M_t} p_{t j} \Big[  \mathfrak{Q}_t ( x_{t-1}^k, 1, \xi_{t j}, 0 )  + \langle \gamma_{t j}^k, \cdot - x_{t-1}^k \rangle        \Big],\\
& = & \theta_t^k + \langle \beta_t^k , \cdot \rangle \mbox{ by definition of }\theta_t^k, \beta_t^k,
\end{array}
\end{equation}
where for the second inequality we have used the subgradient inequality and the definition of $\beta_{t j}^k, \gamma_{t j}^k$.
Combining \eqref{induction}, \eqref{lboundthetabeta}, and the relation $\mathcal{Q}_t^k(\cdot,1) = \max(\mathcal{Q}_t^{k-1}(\cdot,1), \theta_t^k + \langle \beta_t^k , \cdot \rangle )$, we obtain \eqref{wanttoshow}, which achieves the induction step and the proof of (i).
\par (ii) It suffices to  observe that the optimal value of \eqref{reformpb1} is the optimal value of \eqref{optvalTstomu} and that, due to (i),
$\mathcal{Q}_2(x_1,1) \geq \mathcal{Q}_2^k(x_1,1)$ (recall that under our assumptions $\mathcal{Q}_2$ does not depend on $\xi_1$).
\par (iii) can be proved following the convergence proofs of SDDP from \cite{philpot} in the linear case and from \cite{guiguessiopt2016} in the nonlinear case
which apply under our assumptions.\hfill
\end{proof}
In the steps of SDDP-TSto above, we have not detailed the computation of $\beta_{t j}^k$ and $\gamma_{t j}^k$. In the linear and nonlinear settings 
mentioned above, the formulas for these coefficients are given below. When $\xi_t=\xi_{t j}$, we will denote by $A_{t j}, B_{t j}$, and $b_{t j}$ the realizations
of $A_t, B_t$, and $b_t$, respectively.\\

\par {\textbf{Computation of $\beta_{t j}^k$ and $\gamma_{t j}^k$ in the nonlinear case.}} 
Formulas for cuts computed by SDDP when $X_t$ is of form \eqref{xtnonlin} were given in Lemma 2.1 in \cite{guiguessiopt2016}.
We recall these formulas below.
For the 
optimization problem 
$$
{\underline{\mathfrak{Q}}}_t^k (  x_{t-1}^k, 1 , \xi_{t j}    , 1     ) = 
\left\{ 
\begin{array}{l}
\inf_{x_t} \;f_t(x_t, x_{t-1}^k , \xi_{t j} ) + \mathcal{Q}_{t+1}^k ( x_{t}, 1  )   \\ 
A_{t j} x_{t} + B_{t j} x_{t-1}^k = b_{t j},\;\;[\lambda_{t j}^{k 1}]\\
g_t(x_t, x_{t-1}^k, \xi_{t j}) \leq 0,\;\hspace*{0.5cm}[\mu_{t j}^{k 1}]\\
x_t \in \mathcal{X}_t,
\end{array}
\right.
$$
denote by $x_{t j}^{k 1}$ an optimal solution,
consider the Lagrangian 
$$
L(x_t, \lambda , \mu ; x_{t-1}^k , \xi_{t j} )= 
f_t(x_t, x_{t-1}^k , \xi_{t j} ) + \mathcal{Q}_{t+1}^k ( x_{t}, 1  ) + 
\lambda^T ( b_{t j} - A_{t j} x_{t} - B_{t j} x_{t-1}^k ) + \mu^T g_t(x_t, x_{t-1}^k, \xi_{t j}),
$$
and optimal Lagrange multipliers $(\lambda_{t j}^{k 1} , \mu_{t j}^{k 1} )$.
Similarly, for the 
optimization problem 
$$
\mathfrak{Q}_t (  x_{t-1}^k, 1 , \xi_{t j}    , 0     ) = 
\left\{ 
\begin{array}{l}
\inf_{x_t} \;f_t(x_t, x_{t-1}^k , \xi_{t j} )    \\ 
A_{t j} x_{t} + B_{t j} x_{t-1}^k = b_{t j},\;\;[\lambda_{t j}^{k 2}]\\
g_t(x_t, x_{t-1}^k, \xi_{t j}) \leq 0,\;\hspace*{0.5cm}[\mu_{t j}^{k 2}]\\
x_t \in \mathcal{X}_t,
\end{array}
\right.
$$
denote by $x_{t j}^{k 2}$ an optimal solution,
consider the Lagrangian 
$$
L(x_t, \lambda , \mu ; x_{t-1}^k , \xi_{t j} )= 
f_t(x_t, x_{t-1}^k , \xi_{t j} )  + 
\lambda^T ( b_{t j} - A_{t j} x_{t} - B_{t j} x_{t-1}^k ) + \mu^T g_t(x_t, x_{t-1}^k, \xi_{t j}),
$$
and optimal Lagrange multipliers $(\lambda_{t j}^{k 2} , \mu_{t j}^{k 2} )$. 

Let $f'_{t, x_{t-1}}( x_{t j}^{k 1} , x_{t-1}^k , \xi_{t j} )$ (resp. $f'_{t, x_{t-1}}( x_{t j}^{k 2} , x_{t-1}^k , \xi_{t j} )$)
be a subgradient of convex function \\
$f_{t}( x_{t j}^{k 1},\cdot,\xi_{t j} )$ (resp. $f_{t}( x_{t j}^{k 2} , \cdot , \xi_{t j} )$)   at $x_{t-1}^k$.
Let $g'_{t,i, x_{t-1}}( x_{t j}^{k 1} , x_{t-1}^k , \xi_{t j} )$ (resp. $g'_{t,i, x_{t-1}}( x_{t j}^{k 2} , x_{t-1}^k , \xi_{t j} )$)
be a subgradient of convex function $g_{t,i}( x_{t j}^{k 1} , \cdot , \xi_{t j} )$ (resp. $g_{t,i}( x_{t j}^{k 2} , \cdot , \xi_{t j} )$)   at $x_{t-1}^k$.
With this notation, setting
$$
\begin{array}{lll}
\beta_{t j}^k = \displaystyle  f'_{t, x_{t-1}}( x_{t j}^{k 1} , x_{t-1}^k , \xi_{t j} )- B_{t j}^T \lambda_{t j}^{k 1} + \sum_{i=1}^m \mu_{t j}^{k 1}( i ) g'_{t,i, x_{t-1}}( x_{t j}^{k 1} , x_{t-1}^k , \xi_{t j} ),\\
\gamma_{t j}^k = \displaystyle f'_{t, x_{t-1} }( x_{t j}^{k 2} , x_{t-1}^k , \xi_{t j} ) - B_{t j}^T \lambda_{t j}^{k 2} + \sum_{i=1}^m \mu_{t j}^{k 2}( i ) g'_{t, i, x_{t-1}}( x_{t j}^{k 2} , x_{t-1}^k , \xi_{t j} ),
\end{array}
$$
then $\beta_{t j}^k$ is a subgradient of  ${\underline{\mathfrak{Q}}}_t^k ( \cdot, 1, \xi_{t j}, 1 )$ at $x_{t-1}^k$
and $\gamma_{t j}^k$ is a subgradient  of  $\mathfrak{Q}_t  ( \cdot , 1, \xi_{t j}, 0 )$ at $x_{t-1}^k$.\\
\par {\textbf{Computation of $\beta_{t j}^k$ and $\gamma_{t j}^k$ in the linear case.}}
Formulas for the cuts in the linear case are well known.
Due to (H1-L) the optimal value of the linear program
$$
{\underline{\mathfrak{Q}}}_t^k (  x_{t-1}^k, 1 , \xi_{t j}    , 1     ) = 
\left\{ 
\begin{array}{l}
\inf_{x_t} \;c_{t j}^T x_t + \mathcal{Q}_{t+1}^k ( x_{t}, 1  )   \\ 
A_{t j} x_{t} + B_{t j} x_{t-1}^k = b_{t j},\;\;[\lambda_{t j}^{k 1}]\\
x_t \geq 0,
\end{array}
\right.
$$
is the optimal value of the corresponding dual problem: 
\begin{equation}\label{duallp1}
{\underline{\mathfrak{Q}}}_t^k (  x_{t-1}^k, 1 , \xi_{t j}    , 1     ) = 
\left\{ 
\begin{array}{l}
\sup_{\lambda , \mu } \;\lambda^T ( b_{t j}  - B_{t j} x_{t-1}^k ) + \sum_{i=1}^k \mu_i \theta_{t+1}^i  \\ 
A_{t j}^T  \lambda  + \sum_{i=1}^k \mu_i \beta_{t+1}^i \leq c_{t j}, \;\sum_{i=1}^k \mu_i =1, \\
\mu_i \geq 0, i=1,\ldots,k.
\end{array}
\right.
\end{equation}
Let $(\lambda_{t j}^{k 1}, \mu_{t j}^{k 1})$ be an optimal solution of dual problem \eqref{duallp1}.

Similarly, due to (H1-L) the optimal value of the linear program
$$
\mathfrak{Q}_t (  x_{t-1}^k, 1 , \xi_{t j}    , 0     ) = 
\left\{ 
\begin{array}{l}
\inf_{x_t} \;c_{t j}^T x_t    \\ 
A_{t j} x_{t} + B_{t j} x_{t-1}^k = b_{t j},\;\;[\lambda_{t j}^{k 2}]\\
x_t \geq 0,
\end{array}
\right.
$$
is the optimal value of the dual problem 
\begin{equation}\label{duallp2}
\mathfrak{Q}_t (  x_{t-1}^k, 1 , \xi_{t j}    , 0     ) = 
\left\{ 
\begin{array}{l}
\sup_{\lambda } \;\lambda^T ( b_{t j}  - B_{t j} x_{t-1}^k ) \\ 
A_{t j}^T  \lambda  \leq c_{t j}.
\end{array}
\right.
\end{equation}
Let $(\lambda_{t j}^{k 2})$ be an optimal solution of dual problem \eqref{duallp2}.

With this notation, setting
$$
\begin{array}{lll}
\beta_{t j}^k = - B_{t j}^T \lambda_{t j}^{k 1} \mbox{ and }
\gamma_{t j}^k = - B_{t j}^T \lambda_{t j}^{k 2},
\end{array}
$$
then $\beta_{t j}^k$ is a subgradient of  ${\underline{\mathfrak{Q}}}_t^k ( \cdot, 1, \xi_{t j}, 1 )$ at $x_{t-1}^k$
and $\gamma_{t j}^k$ is a subgradient  of  $\mathfrak{Q}_t  ( \cdot , 1, \xi_{t j}, 0 )$ at $x_{t-1}^k$.

\if{

\subsection{An alternative SDDP algorithm}

Since when the optimization period ends no future costs are incurred, it is sufficient
to stop the forward pass when we reach the final stage $T$, i.e., to compute decisions
$x_t^k,t=1,\ldots,T$, for iteration $k$.
Recalling Remark \ref{remsimplSDDPTSto}, we will still replace functions $\mathcal{Q}_t(\cdot,0)$ by the null function
and will denote by $\mathcal{Q}_{t}^k(\cdot)$ the approximation of $\mathcal{Q}_t(\cdot,1)$ obtained at  the
end of iteration $k$.
This approximation now takes the form
\begin{equation}\label{newapproxsddp}
\mathcal{Q}_t^k(x_{t-1}) = \displaystyle \max_{0 \leq j \leq N_{t}^k} \;\theta_t^j + \langle \beta_t^j , x_{t-1} \rangle.
\end{equation}
Note that for a given iteration $k$, the number of cuts $N_{t}^k$ is now not necessarily the same for all stages $t=2,\ldots,T$:   for two cost-to-go functions 
$\mathcal{Q}_{t_1}^k(x_{t-1},1)$ and $\mathcal{Q}_{t_2}^k(x_{t-1},1)$ with $t_1 \neq t_2$ the number of cuts computed 
$N_{t_1}^k$ and $N_{t_2}^k$ up to iteration $k$ can be different. This comes from the fact that the number of stages, and therefore the stage until which decisions are computed
at a given iteration, is random. The corresponding variant of SDDP-TSto, denoted SDDP'-TSto is given below.\\

\par {\textbf{SDDP'-TSto, Step 1: Initialization.}} Fix an integer $N>1$. For $t=2,\ldots,T_{\max}$, take for $\mathcal{Q}_t^0(\cdot)=\theta_t^0 + \langle \beta_t^0 , x_{t-1} \rangle$ 
a known lower bounding affine function for $\mathcal{Q}_t(\cdot,1)$.
Set the iteration count $k$ to 1, $\mathcal{Q}_{T_{\max}+1}^0(\cdot)  \equiv 0$, and set variable
$N_t^0$ (this variable will store the current number of cuts computed for $\mathcal{Q}_t(\cdot, 1)$) to $0$ for $t=2,\ldots,T_{\max}$.\\
\par {\textbf{SDDP'-TSto, Step 2: Forward pass.}} 
We generate a sample
$$ 
(  (\tilde \xi_1^k ,  \tilde D_1^k ),(\tilde \xi_2^k ,  \tilde D_2^k ),\ldots,(\tilde \xi_{T_{\max}}^k ,  \tilde D_{T_{\max}}^k )),
$$
from the distribution of $((\xi_1, D_1 ),(\xi_2, D_2),\ldots,(\xi_{T_{\max}}, D_{T_{\max}} ))$,
with the convention that $\tilde \xi_1^k = \xi_1$, $\tilde D_1^k=1$. \\
Cost$_k =$ 0, $t=1$.\\
{\textbf{While }}$\tilde D_t^k=1$, we compute an optimal solution $x_t^k$ of 
\begin{equation}\label{pbforwardpass}
\left\{
\begin{array}{l}
\inf_{x_t \in \mathbb{R}^n}  f_t(x_t , x_{t-1}^k, \tilde \xi_t^k  ) + \mathcal{Q}_{t+1}^{k-1} ( x_t  )\\
x_t \in X_t(x_{t-1}^k, {\tilde \xi}_{t}^k )
\end{array}
\right.
\end{equation}
\hspace*{1.2cm}where $x_0^k=x_0$.\\
\hspace*{1.2cm}Cost$_k \leftarrow $Cost$_k + f_t(x_t^k , x_{t-1}^k, \tilde \xi_t^k  )$.\\
\hspace*{1.2cm}$t \leftarrow t+1$.\\
{\textbf{End While}}\\
We compute an optimal solution $x_t^k$ of 
\begin{equation}\label{pbforwardpass1}
\left\{
\begin{array}{l}
\inf_{x_t \in \mathbb{R}^n}  f_t(x_t , x_{t-1}^k, \tilde \xi_t^k  ) \\
x_t \in X_t(x_{t-1}^k, {\tilde \xi}_{t}^k ).
\end{array}
\right.
\end{equation}
Cost$_k \leftarrow $Cost$_k + f_t(x_t^k , x_{t-1}^k, \tilde \xi_t^k  )$.\\
\par {\textbf{Upper bound computation:}} If $k \geq N$ compute
$$
{\overline{\mbox{Cost}}}_k = \frac{1}{N} \sum_{j=k-N+1}^k \mbox{Cost}_j,\;\;{\hat \sigma}_{N,k}^2= \frac{1}{N}\sum_{j=k-N+1}^k [\mbox{Cost}_j - {\overline{\mbox{Cost}}}_k ]^2,
$$
and the upper bound
$$
{\overline U}_k = {\overline{\mbox{Cost}}}_k + \frac{{\hat \sigma}_{N,k}}{\sqrt{N}} t_{N-1, 1-\alpha}
$$
where $t_{N-1, 1-\alpha}$ is the $(1-\alpha)$-quantile of the Student distribution with $N-1$ degrees of freedom.\\
\par {\textbf{SDDP'-TSto, Step 3: Backward pass.}}\\
Let ${\underline{\mathfrak{Q}}}_t^k( x_{t-1} , D_{t-1} , \xi_t, D_t )$ be the function given by
\begin{equation}\label{dp7knewbis}
{\underline{\mathfrak{Q}}}_t^k (  x_{t-1}, D_{t-1} , \xi_t    , D_{t}     ) = 
\left\{ 
\begin{array}{l}
\inf_{x_t} \;D_{t-1} f_t(x_t, x_{t-1}, \xi_t ) + \mathcal{Q}_{t+1}^k ( x_{t} )   \\ 
x_t \in X_t(x_{t-1}, \xi_t ).
\end{array}
\right.
\end{equation}
Note that functions $\mathcal{Q}_{t+1}^k$ in the right-hand-sides of \eqref{dp7knew} and \eqref{dp7knewbis} are not the same:
in  \eqref{dp7knew} this function is given by \eqref{sddpapprox} with $t$ replaced by $t+1$ and in
\eqref{dp7knewbis} it is given by \eqref{newapproxsddp} with $t$ replaced by $t+1$ (in these formulas both the sequence of coefficients $(\theta_t^k, \beta_t^k)_k$ 
and the number of affine pieces can differ).
Therefore ${\underline{\mathfrak{Q}}}_t^k$ defined in \eqref{dp7knewbis} is not the same as 
${\underline{\mathfrak{Q}}}_t^k$ given by \eqref{dp7knew}.\\
{\textbf{For}} $\tau=t$ down to $\tau=2$, \\
\hspace*{0.7cm}$N_\tau^k = N_{\tau}^{k-1} +1$.\\
\hspace*{0.7cm}{\textbf{For }}$j=1,\ldots,M_t$,\\
\hspace*{1.4cm}Compute ${\underline{\mathfrak{Q}}}_\tau^k ( x_{\tau-1}^k, 1, \xi_{\tau j}, 1 )$, compute
$$
\mathfrak{Q}_t  ( x_{\tau-1}^k, 1, \xi_{\tau j}, 0 )=
\left\{ 
\begin{array}{l}
\inf_{x_\tau} \;f_{\tau}(x_{\tau}, x_{\tau-1}^k, \xi_{\tau j} )  \\ 
x_{\tau} \in X_{\tau}(x_{\tau-1}^k, \xi_{\tau j} ),
\end{array}
\right.
$$
\hspace*{1.4cm}compute a subgradient $\beta_{\tau j}^k$ of  ${\underline{\mathfrak{Q}}}_\tau^k ( \cdot, 1, \xi_{\tau j}, 1 )$ at $x_{\tau-1}^k$
and a subgradient $\gamma_{\tau j}^k$\\
\hspace*{1.4cm}of $\mathfrak{Q}_\tau  ( \cdot , 1, \xi_{\tau j}, 0 )$ at $x_{\tau-1}^k$.\\
\hspace*{0.7cm}{\textbf{End For}}\\
\hspace*{0.7cm}Compute
$$
\begin{array}{lll}
\theta_\tau^{N_\tau^k } &= &(1-q_\tau)  \displaystyle  \sum_{j=1}^{M_\tau} p_{\tau j} \Big( {\underline{\mathfrak{Q}}}_\tau^k ( x_{\tau-1}^k, 1, \xi_{\tau j}, 1 )   -  \langle \beta_{\tau j}^k , x_{\tau-1}^k  \rangle   \Big) \\
&&+ q_\tau  \displaystyle \sum_{j=1}^{M_\tau} p_{\tau j} \Big( \mathfrak{Q}_\tau ( x_{\tau-1}^k, 1, \xi_{\tau j}, 1 )   -  \langle \gamma_{\tau j}^k , x_{\tau-1}^k  \rangle   \Big),\\
\beta_\tau^{ N_\tau^k } &=& (1-q_\tau) \displaystyle  \sum_{j=1}^{M_\tau} p_{\tau j} \beta_{\tau j}^k + q_\tau  \sum_{j=1}^{M_\tau} p_{\tau j} \gamma_{\tau j}^k,
\end{array}
$$
\hspace*{0.7cm}where $\gamma_{\tau j}^k$ and $\beta_{\tau j}^k$ are computed as explained above.\\
{\textbf{End For}}\\
{\textbf{For}} $\tau=t+1,\ldots,T_{\max}$\\
\hspace*{0.7cm}$N_\tau^k = N_{\tau}^{k-1}$.\\
{\textbf{End For}}\\
\par \par {\textbf{Lower bound computation:}} compute the lower bound $\underline{L}_k$ on $\mathcal{Q}_1( x_0 )$ given by
$$
\underline{L}_k =
\left\{ 
\begin{array}{l}
\inf_{x_1} \;f_1(x_1, x_{0}, \xi_1 ) + \mathcal{Q}_{2}^k ( x_{1})   \\ 
x_1 \in X_1(x_{0}, \xi_1 ).
\end{array}
\right.
$$
\par {\textbf{SDDP'-TSto, Step 4:}} If $\frac{{\overline U}_k - \underline{L}_k}{{\overline U}_k} \leq \mbox{Tol}$ stop otherwise do $k \leftarrow k+1$ and go to Step 2.\\

}\fi

\section{Numerical experiments: portfolio selection with a random investment period}\label{sec:numexp}

\subsection{SDDP-TSto for the portfolio problem} \label{sec:sddpportf}

We consider the portfolio problem given in Section  \ref{simpleportpb}
and the corresponding dynamic programming equations
\eqref{dp1simpterport}, \eqref{dp7terport}, \eqref{dp7terport2}
when the number of stages is stochastic.  
We now specialize the SDDP-TSto algorithm applied to these DP equations. In particular,
we explain how to compute coefficients $\gamma_{t j}^k$ and $\beta_{t j}^k$ for this problem.

In the forward pass of SDDP-TSto, problem \eqref{pbforwardpassb} becomes 
\begin{equation}\label{pbforwardpass}
\left\{
\begin{array}{l}
\inf_{x_t \in \mathbb{R}^n} \mathcal{Q}_{t+1}^{k-1} ( x_t , 1  )\\
x_t \in X_t(x_{t-1}^k, {\tilde \xi}_{t}^k )
\end{array}
\right.
\end{equation}
while problem \eqref{pbforwardpass1} specializes to
\begin{equation}\label{pbforwardpass1port}
\left\{
\begin{array}{l}
\inf_{x_t \in \mathbb{R}^n}  -\mathbb{E}[ \sum\limits_{i=1}^{n+1} \xi_{t+1}(i) x_{t}(i) ] \\
x_t \in X_t(x_{t-1}^k, {\tilde \xi}_{t}^k ).
\end{array}
\right.
\end{equation}
The backward pass is given below.\\
\par {\textbf{Backward pass of SDDP-TSto for the portfolio problem.}}
Set $\mathcal{Q}_{T_{\max}+1}^k(\cdot,1)\equiv 0$.\\
{\textbf{For}} $t=T_{\max}$ down to $t=2$, \\
\hspace*{0.7cm}{\textbf{For }}$j=1,\ldots,M_t$,\\
\hspace*{1.4cm}Solve the optimization problem 
$$
\left\{
\begin{array}{l}
\inf -\mathbb{E}[ \sum\limits_{i=1}^{n+1} \xi_{t+1}(i) x_{t}(i) ] \\
x_t \in X_t ( x_{t-1}^k , {\xi}_{t j} ),
\end{array}
\right.
$$
\hspace*{1.4cm}with Lagrangian $L(x_t,y_t,z_t,\lambda_1 ,\mu_1, \delta_1 )$ given by
$$
\begin{array}{l}
 -\mathbb{E}[ \sum\limits_{i=1}^{n+1} \xi_{t+1}(i) x_{t}(i) ]
+ \left  \langle \lambda_1 , \xi_{t j} \circ  x_{t-1}-x_t  + \left[
\begin{array}{c}
-y_t + z_t  \\
 ({\textbf{e}} - \eta_t  )^T y_t - ({\textbf{e}} + \nu_t  )^T z_t 
\end{array}
\right] \right \rangle \\
+ \left  \langle  \mu_{1} ,  y_{t}   -  \xi_{t j}(1:n) \circ  x_{t-1}(1;n) \right \rangle + 
\left   \langle \delta_{1}  , x_{t}(1:n) - (\xi_{t j}^T x_{t-1}) u \right  \rangle 
\end{array}
$$
\hspace*{1.4cm}where ${\textbf{e}}$ is a vector in $\mathbb{R}^{n}$ of ones, $\lambda_1 \in \mathbb{R}^{n+1}, \mu_1, \delta_1 \in \mathbb{R}^n$, and for vectors $x, y$, \\
\hspace*{1.4cm}the vector  $x \circ y$ has components $(x \circ y)(i)=x(i) y(i)$ and $\langle x , y \rangle = x^T y$.\\ 
\hspace*{1.4cm}For this Lagrangian, let $(\lambda_{1 t j}^k, \mu_{1 t j}^k, \delta_{1 t j}^k)$ be optimal Lagrange multipliers.\\
\hspace*{1.4cm}Solve the optimization problem 
$$
\left\{
\begin{array}{l}
\inf \mathcal{Q}_{t+1}^k ( x_t , 1) \\
x_t \in X_t ( x_{t-1}^k , {\xi}_{t j} ),
\end{array}
\right.
$$
\hspace*{1.4cm}with Lagrangian $L(x_t, y_t, z_t, \lambda_2 ,\mu_2, \delta_2 )$ given  by
$$
\begin{array}{l}
\mathcal{Q}_{t+1}^k ( x_t , 1) + \left  \langle \lambda_2 , \xi_{t j} \circ  x_{t-1}-x_t  + \left[
\begin{array}{c}
-y_t + z_t  \\
({\textbf{e}} - \eta_t  )^T y_t - ({\textbf{e}} + \nu_t  )^T z_t 
\end{array}
\right] \right \rangle \\
+ \left  \langle  \mu_{2} ,  y_{t}   -  \xi_{t j}(1:n) \circ  x_{t-1}(1;n) \right \rangle + 
\left   \langle \delta_2  , x_{t}(1:n) - (\xi_{t j}^T x_{t-1}) u \right  \rangle 
\end{array}
$$
\hspace*{1.4cm}where ${\textbf{e}}$ is a vector in $\mathbb{R}^{n}$ of ones, $\lambda_2 \in \mathbb{R}^{n+1}, \mu_2, \delta_2 \in \mathbb{R}^n$. \\
\hspace*{1.4cm} For this Lagrangian, let $(\lambda_{2 t j}^k, \mu_{2 t j}^k, \delta_{2 t j}^k)$ be optimal Lagrange multipliers.\\
\hspace*{1.4cm}Compute
$$
\begin{array}{l}
\gamma_{t j}^k =  \Big( \lambda_{1 t j}^k - ( u^T \delta_{1 t j}^k ) {\textbf{e}} - \left[ \begin{array}{c}\mu_{1 t j}^k\\0  \end{array}   \right]  \Big) \circ \xi_{t j},\\
\beta_{t j}^k = \Big(
\lambda_{2 t j}^k - ( u^T \delta_{2 t j}^k ) {\textbf{e}} - \left[ \begin{array}{c}\mu_{2 t j}^k\\0  \end{array}   \right]  \Big) \circ \xi_{t j},
\end{array}
$$
\hspace*{1.4cm}where ${\textbf{e}}$ is a vector in $\mathbb{R}^{n+1}$ of ones.\\
\hspace*{0.7cm}{\textbf{End For}}\\
\hspace*{0.7cm}Compute\footnote{Observe that the intercept for the cuts is zero for that application.}
$$
\begin{array}{l}
\theta_{t}^{k }=0 \mbox{ and }
\beta_{t}^{k } = (1-q_t) \displaystyle \sum_{j=1}^{M_t} p_{t j} \beta_{t j}^k + q_{t}\displaystyle \sum_{j=1}^{M_t} p_{t j}  \gamma_{t j}^k.
\end{array}
$$
{\textbf{End For}}

\subsection{Numerical results}\label{sec:numresults}

Our goal in this section is to compare SDDP run on an optimization period of $T_{\max}$ stages
(i.e., where we assume that the number of stages is known and fixed to $T_{\max}$) and SDDP-TSto 
on the risk-neutral portfolio problem with direct transaction costs given in Section \ref{simpleportpb}. 
All subproblems in the forward and backward passes of SDDP and SDDP-TSto were solved numerically using 
the interior point solver of the Mosek Optimization Toolbox \cite{mosek} and the Matlab code
was run on a Xeon E5-2670 processor with  384 GB of RAM. 
The following parameters are chosen for our experiments.\\

\par {\textbf{Distributions of $T$ and of returns $(\xi_t)$.}}
We consider two distributions for random variable $T$: a uniform distribution over 
the set $\{2,3,\ldots,T_{\max}\}$ and a 
truncated exponential discretized distribution.
The reason for choosing this latter distribution for a portfolio problem is motivated by
reference \cite{lifetimecompanies2015}, where the lifetime of more than 25 000 publicly traded 
North American companies, from 1950 to 2009, 
was analyzed. It was shown that {\em{mortality rates are independent of the company's age, the typical half-life of a publicly traded company is about a decade}},
and the exponential distribution is a good fit for the lifetime of these companies on this period.
Therefore, for such companies, the exponential distribution makes sense for the duration of an optimization period of a portfolio problem.
However, since the number of stages is almost surely in the set $\{2,\ldots,T_{\max}\}$, instead of an exponential distribution, we take for $T$ 
a translation of a
discretization of an exponential distribution conditioned on the event that this exponential distribution belongs to
$[\frac{1}{2},T_{\max}-\frac{1}{2}[$.
In what follows, we call that distribution truncated exponential discretized.\footnote{Considering the values chosen for the realizations of the returns and today's ``usual" asset returns,
we can consider that a stage corresponds to a few years and that the maximal duration of the optimization period $T_{\max}=10$ is a few decades.}
More precisely, let $X \sim \mathcal{E}(\lambda)$ be the exponential distribution with
parameter $\lambda=0.15$ with expectation $\mathbb{E}[X] = \frac{1}{\lambda}\approx 6.67$.
Setting $Y= X | A $ where $A$ is the event $A=\{\omega : \frac{1}{2} \leq X( \omega )  \leq T_{\max}-\frac{1}{2}  \}$
and defining the random variable $T$ by $T=t+1$ if and only if $t - \frac{1}{2} \leq Y < t + \frac{1}{2}$
for $t=1,\ldots,T_{\max}-1$, then the distribution of the number of stages $T$ is given by
$$
\mathbb{P}(T=t+1) = 
\mathbb{P}\Big( t-\frac{1}{2} \leq  Y  <  t + \frac{1}{2}  \Big) = 
\frac{\mathbb{P}(t-\frac{1}{2} \leq X < t + \frac{1}{2})}{\mathbb{P}(A)} = 
\frac{e^{-\lambda ( t-\frac{1}{2} )}   - e^{-\lambda ( t+\frac{1}{2} )}  }{ e^{-\lambda / 2 }   - e^{-\lambda ( T_{\max}- \frac{1}{2}  )}},
$$
for $t=1,\ldots,T_{\max}-1$. The histogram  of the distribution of $T$ when $T_{\max}=10$
 is represented in Figure \ref{fig:f_0} together with the graph of the density of 
$2+X$ over the interval $[2,10]$.
\begin{figure}
\centering
\includegraphics[scale=0.6]{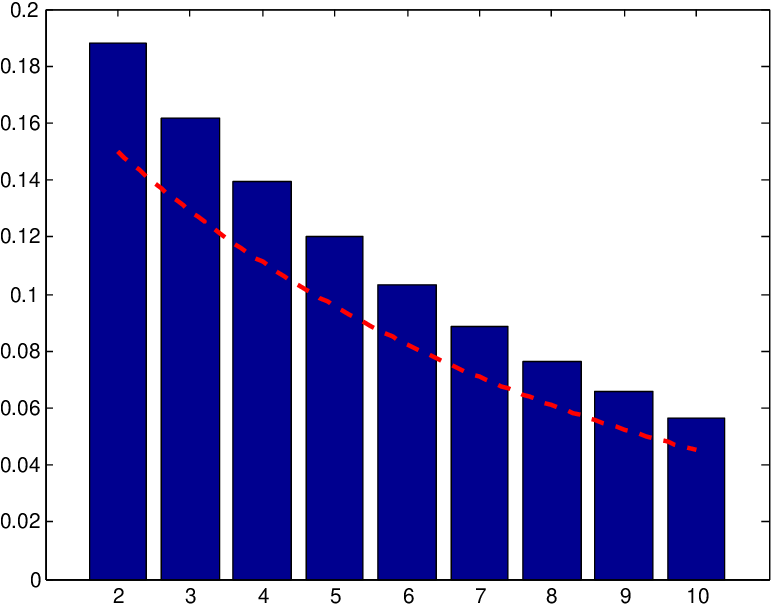}
\caption{ \label{fig:f_0} Histogram of the distribution of $T$ with support $\{2,\ldots,T_{\max}\}=\{2,\ldots,10\}$ and density 
of $2+\mathcal{E}(\lambda)$ (in dotted line) on the interval $[2,10]$ with $\lambda=0.15$.}
\end{figure}

The return of the risk-free asset $n+1$ is 1.01 for every stage.
Returns $\xi_t,t=2,\ldots,T_{\max}$, have discrete distributions with $M_t=M$ realizations, each having probability $\frac{1}{M}$ and realizations
$\xi_1(1:n), \xi_{t 1}(1:n), \ldots, \xi_{t M}(1:n)$ obtained sampling
from a normal distribution with mean and standard deviation chosen 
randomly in respectively the intervals $[0.9,1.4]$ and $[0.1,0.2]$.\\

\par {\textbf{Remaining parameters of the portfolio problem.}}
The initial portfolio $x_0$ has components $x_{0}(i), i=1,\ldots,n+1$, uniformly distributed in $[0,1000]$ (vector of initial wealth in each asset).
Transaction costs are known with
$\nu_t(i)=\mu_t(i)$
obtained sampling from the distribution of the random variable 
$0.08+0.06\cos(\frac{2\pi}{T} U_{T_{\max}} )$ where $U_{T_{\max}}$ is a random variable
with a discrete distribution over the set of integers $\{1,2,\ldots,T_{\max}\}$. The largest position in any security is set to $80\%$, i.e., $u(i)=0.8$ for $i=1,\ldots,n$.\\

\par {\textbf{Parameters of SDDP and SDDP-TSto methods.}} 
Using the notation of the previous section, SDDP and SDDP-TSto are run with 
parameters $N=400$, $\alpha=0.05$, and Tol=0.05.\\

\par {\textbf{Comparing SDDP and SDDP-TSto policies on Monte-Carlo simulations.}}\\

\par We consider 6 instances of the portfolio problem with random number of stages $T$ that we have just presented where the distribution of   
$T$ is either the truncated exponential discretized 
distribution given above 
or 
the uniform distribution on the set 
$\{2\ldots,T_{\max}\}$.
For each distribution of $T$,
we take 3 values for $(T_{\max},n,M)$
given in Table \ref{tableoptvalues}.
This defines  6 instances for which  we compute
two policies.

\par The first policy, denoted by SDDP, is obtained running SDDP on the portfolio problem
assuming, as has usually been the case so far in applications of stochastic programming, that the number
of stages is fixed. Since we need a policy defined for all possible stages from $t=1$
to $t=T_{\max}$, this SDDP policy is obtained running SDDP on a portfolio problem with a deterministic
number of stages fixed to $T_{\max}$. 
\par The second policy is obtained running the SDDP-Tsto algorithm from Section \ref{sec:algosddptsto} on the portfolio problem 
with random number of stages $T$. With a slight abuse of notation, we call the corresponding policy SDDP-TSto.

\par For each instance, we then simulate SDDP and SDDP-TSto policies
on 5 000 scenarios for $((D_1,\xi_1),$ $(D_2,\xi_2)$,$\ldots$,$(D_{T_{\max}},\xi_{T_{\max}}))$.
For each instance, the empirical mean income with both policies on these 5 000 scenarios 
is given in Table \ref{tableoptvalues}. We see that, as expected, the mean income is larger (sometimes much larger) with 
SDDP-TSto. We also report in this table the $p$-value of the paired Student  t-test
\begin{equation}\label{pairedttest}
H_0: \mathbb{E}[\displaystyle I_{\mbox{SDDP}}] \geq  \mathbb{E}[I_{\mbox{SDDP-TSto}}]
\mbox{ against }H_1: \mathbb{E}[I_{\mbox{SDDP}}] < \mathbb{E}[I_{\mbox{SDDP-TSto}}],
\end{equation}
computed using the 5 000 realizations of income of
SDDP-TSto and SDDP policies.
In this test, $I_{\mbox{SDDP}}$ and $I_{\mbox{SDDP-TSto}}$ are the income 
with respectively SDDP and SDDP-TSto policies.
We observe that for the three instances the $p$-value is small.
More precisely, on four instances the $p$-value is less than 5\%
meaning that at
the significance level $5\%$, we accept for these instances the hypothesis $H_1$, i.e., the hypothesis that the 
mean income with 
SDDP-TSto is greater than the mean income with SDDP policy.
For the remaining two instances, the $p$-values are 6.8\% and
7.2\%, i.e., we accept for these 2 instances the hypothesis that SDDP-TSto income 
is higher than SDPP income at any significance level above 7.2\%.

\begin{rem}
We performed the same Monte Carlo simulations on several other instances
with larger values of $n$ and $M$. For all our experiments, the empirical mean income with 
SDDP-TSto was higher than the empirical mean income with SDDP
but the $p$-value of paired t-test \eqref{pairedttest}
could be as high as 0.4. From these experiments, it seems
that the larger the size of the deterministic equivalent 
(in particular the larger $n$ and $M$), the larger the standard deviation
of the income and therefore the less paired $t$-test \eqref{pairedttest}
 will accept  (at the significance level 5\%) the hypothesis that
 SDDP-TSto income is higher.
\end{rem}

\begin{table}
\begin{tabular}{|c|c|c|c|c|c|c|}
\hline
Distribution of $T$& $T_{\max}$ & $n$ & $M$   & \begin{tabular}{c}Empirical mean\\income\\SDDP-TSto\end{tabular} & \begin{tabular}{c}Empirical mean\\income\\SDDP\end{tabular} & $p$-value\\
\hline
 \begin{tabular}{c}Uniform\\on $2,\ldots,T_{\max}$\end{tabular} &  10 &  10 & 50 &   14 973 &   14 563  &    0.035 \\
\hline 
\begin{tabular}{c}Truncated\\exponential discretized\end{tabular} & 10 & 10    & 50 & 12  466   & 12 138    &  0.002   \\
\hline 
 \begin{tabular}{c}Uniform\\on $2,\ldots,T_{\max}$\end{tabular} & 10 &  20   & 50 &  34 410  & 33 875    &   0.072  \\
\hline 
\begin{tabular}{c}Truncated\\exponential discretized\end{tabular} &10 &  20    & 50 &  32 783  &  31 121    & $6.5\small{\times}10^{-8}$    \\
\hline 
\begin{tabular}{c}Uniform\\on $2,\ldots,T_{\max}$\end{tabular} &  10 & 10  & 100 & 12 977   & 12 797    & 0.068    \\
\hline 
\begin{tabular}{c}Truncated\\exponential discretized\end{tabular} & 20 &  10   & 20 &  24 181  & 23 700    &  0.040   \\
\hline 
\end{tabular}
\caption{Empirical mean income for SDDP-TSto and SDDP policies on 6 instances of the portfolio problem
with a random number of stages and $p$-values of the test 
$H_0: \mathbb{E}[\displaystyle I_{\mbox{SDDP}}] \geq  \mathbb{E}[I_{\mbox{SDDP-TSto}}]$ against 
$H_1: \mathbb{E}[I_{\mbox{SDDP}}] < \mathbb{E}[I_{\mbox{SDDP-TSto}}]$
for these instances, on the basis of samples of size 5 000 of
the income of SDDP-TSto and SDDP policies.
In $H_0$ and $H_1$, $I_{\mbox{SDDP}}$ and $I_{\mbox{SDDP-TSto}}$ are the income 
with respectively SDDP and SDDP-TSto policies.
}\label{tableoptvalues}
\end{table}

\par {\textbf{Comparing the empirical distribution of the solution time for SDDP and SDDP-TSto.}}\\

\par When we face a multistage stochastic program with a random number of stages $T$ with 
support the set of integers $\{2,3,\ldots,T_{\max}\}$, we also expect SDDP-TSto policy
to be computed quicker than SDDP policy, recalling that the latter is obtained running SDDP
assuming the number of stages is fixed and equal to $T_{\max}$.
Indeed, in the forward pass of iteration $k$ of SDDP-TSto, once we reach the first stage $t_0$ such that 
$\tilde D_{t_0}^k = 0$, for subproblems of subsequent  stages $t>t_0$, the objective function
is null, i.e., all that  the trial points have to satisfy is to be feasible.
In the case of our portfolio problem, we can even obtain a drastic reduction in the time
needed to compute SDDP-TSto policy as can be seen on Figure \ref{fig:f_1}
which reports the empirical distribution of the computational time needed to solve 
 instances of the portfolio problem
(with truncated exponential discretized distribution for
$T$) 10 times with SDDP and SDDP-TSto.    
Accross the 10 runs, the time needed to compute a given  policy
does not vary much and the computational bulk with SDDP-TSto is around 10 times
inferior to the computational bulk with SDDP.\\


\begin{figure}
\centering
\begin{tabular}{cc}
\includegraphics[scale=0.6]{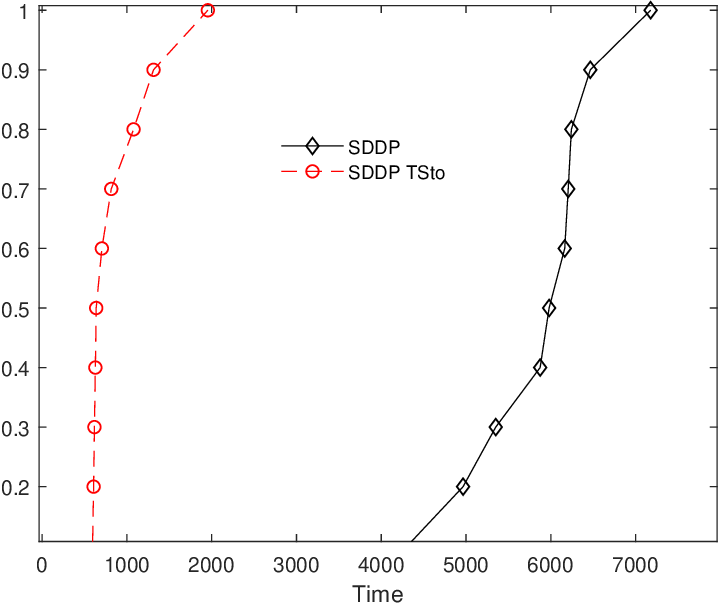}
&
\includegraphics[scale=0.6]{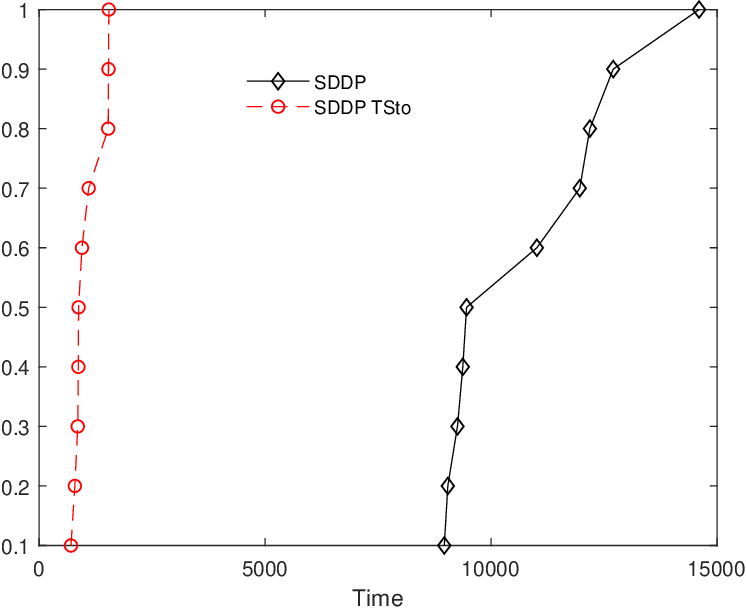}\\
$M=50, n=10, T_{\max}=10$ &  $M=50, n=20, T_{\max}=10$      \\
&\\
\end{tabular}\\
\begin{tabular}{cc}
\includegraphics[scale=0.6]{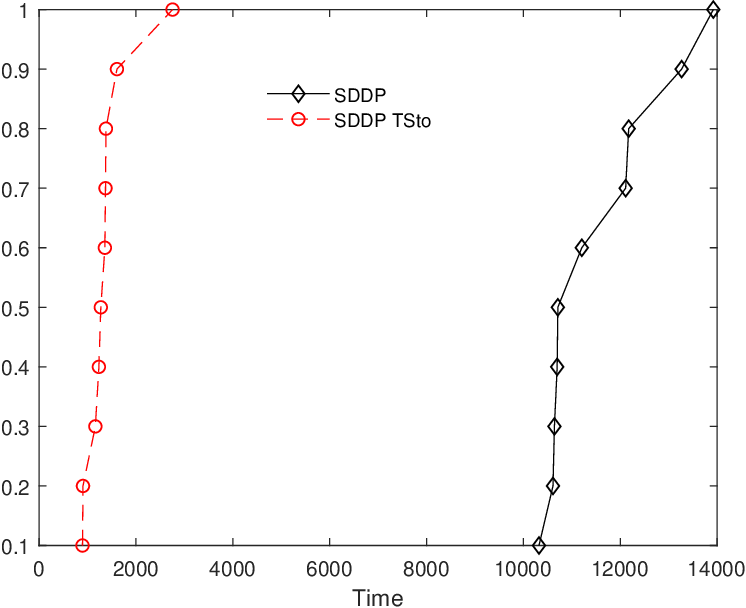}\\
$M=100, n=10, T_{\max}=10$ \\
\end{tabular}\\
\caption{Empirical distribution of the computational time needed to solve instances of the portfolio problem
with truncated exponential discretized distribution for
$T$ (for each instance, SDDP and SDDP-TSto are run 10 times).}\label{fig:f_1}
\end{figure}

\par {\textbf{Summary of the advantages of the proposed methodology and solution method.}}\\

\par So far, to our knowledge, stochastic programs with a random number of stages had not been
considered in the stochastic programming community. Therefore, in the presence of a multistage
stochastic program with a random number $T$ of stages upper bounded by $T_{\max}$, the traditional approach would be 
in fact to ignore the stochasticity of $T$, to
assume that the number of stages is fixed to $T_{\max}$ so that a policy can be defined for all possible
stages $t=1,2,\ldots,T_{\max}$, and, as long as the problem is convex, a popular approach
to solve such a problem would be to use SDDP. Compared to this approach, the tools proposed in this article
offer the following advantages:
\begin{itemize}
\item The Dynamic Programming equations written in Section \ref{dpequations} define the 
appropriate Bellman functions for the problem under consideration, which, obviously, yield an optimal cost lower
than the one obtained using Bellman functions defined by Dynamic Programming equations
for a problem with $T_{\max}$ fixed stages.
Observe also that the model and Dynamic Programming equations of Section \ref{dpequations}
do not require any assumption of convexity.
\item If the stochastic program is convex and
as long as SDDP-TSto algorithm given in Section \ref{sec:sddp} 
and SDDP (run on the Dynamic Programming equations with $T_{\max}$ fixed stages) 
are run for a sufficient amount of iterations to satisfy a stopping
criterion with small values of $\alpha$ and Tol, SDDP-TSto policy will be better (and can be much better)
than SDDP policy, see the experiments reported in this section on the portfolio problem.
\item We expect SDDP-TSto policy  to be computed quicker than SDDP policy (the latter being run 
on a problem with $T_{\max}$ stages).
For our portfolio problem with a random number of stages, SDDP-TSto was computed up to 10 times faster
than SDDP, see Figure \ref{fig:f_1}.
\end{itemize}

\section{Conclusion}

We introduced the class of multistage stochastic programs with a random number of stages.
We explained how to write dynamic programming equations for such problems and detailed the SDDP algorithm
to solve these dynamic programming equations. 
We have shown the applicability and interest of the  proposed models and methodology for portfolio
selection.

As a future work, it would be interesting to consider more general ``hybrid" stochastic programs
with transition probabilities between objective and cost functions, meaning that at each stage
not only parameters but also cost and constraint functions are random, possibly depending on past values
of parameters and cost and constraint functions. 
Other possible extensions could consider stages of random durations or problems where decisions are 
only taken at random times driven by an exogenous stochastic process. The random times at which
these decisions would be taken could also depend on previous decisions.
Finally, it would be interesting to use the proposed models and methodology for other
applications, for instance Asset Liability Management or the applications mentioned in the introduction.

\section*{Acknowledgments} The author's research was 
partially supported by an FGV grant, CNPq grants 311289/2016-9 and 
401371/2014-0, and FAPERJ grant E-26/201.599/2014. The author would like to thank
Alexander Shapiro for helpful discussions.

\addcontentsline{toc}{section}{References}
\bibliographystyle{plain}
\bibliography{SDDP_T_Sto_Conv}

\end{document}